\documentclass[12pt]{amsart}
\usepackage{graphicx}
\usepackage{amsmath}
\usepackage{amsthm}
\usepackage{amsfonts}
\usepackage{amssymb}
\usepackage[usenames,dvipsnames]{color}
\usepackage{verbatim}
\usepackage{amscd}
\usepackage[latin1]{inputenc}
\usepackage{latexsym}
\usepackage{graphicx}
\usepackage{dsfont}
\usepackage{epsfig}

\newtheorem{theorem}{Theorem}

\newtheorem{lemma}{Lemma}

\newtheorem{proposition}[theorem]{Proposition}

\newtheorem*{theorem*}{Theorem}

\theoremstyle{definition}
\newtheorem{definition}{Definition}
\newtheorem{example}{Example}
\newtheorem{remark}{Remark}

\newenvironment{PfofHiddenGibbs}[1]
{\par\vskip2\parsep\noindent{\sc Proof of Theorem\ \ref{Thm:HiddenGibbs}. }}{{\hfill
$\Box$}
\par\vskip2\parsep}

\newenvironment{PfofPosteriorConvergence}[1]
{\par\vskip2\parsep\noindent{\sc Proof of Theorem\ \ref{Thm:PosteriorConvergence}. }}{{\hfill
$\Box$}
\par\vskip2\parsep}

\newenvironment{PfofDirectGibbs}[1]
{\par\vskip2\parsep\noindent{\sc Proof of Theorem\ \ref{Thm:DirectGibbs}. }}{{\hfill
$\Box$}
\par\vskip2\parsep}

\newenvironment{PfofPropLB}[1]
{\par\vskip2\parsep\noindent{\sc Proof of Proposition\ \ref{Prop:LB}. }}{{\hfill
$\Box$}
\par\vskip2\parsep}

\newenvironment{PfofPropUB}[1]
{\par\vskip2\parsep\noindent{\sc Proof of Proposition\ \ref{Prop:UB}. }}{{\hfill
$\Box$}
\par\vskip2\parsep}

\def\N{\mathbb{N}}

\def\R{\mathbb{R}}

\def\cJ{\mathcal{J}}
\def\cM{\mathcal{M}}

\def\1{\mathbf{1}}
\def\cU{\mathcal{U}}
\def\cX{\mathcal{X}}
\def\cX{\mathcal{X}}
\def\cY{\mathcal{Y}}

\def\Id{I_{\Theta}}

\DeclareMathOperator*{\argmin}{argmin}

\DeclareMathOperator{\diam}{diam}

\DeclareMathOperator{\proj}{proj}

\title[Gibbs posteriors and thermodynamics]{Gibbs posterior convergence and the thermodynamic formalism}

\begin{document}

\author{Kevin McGoff}
\address{UNC Charlotte}
\email{kmcgoff1@uncc.edu}
\thanks{Kevin McGoff would like to acknowledge funding from NSF grant DMS 1613261.}

\author{Sayan Mukherjee}
\address{Duke University}
\email{sayan@stat.duke.edu}
\thanks{Sayan Mukherjee would like to acknowledge funding from NSF  DEB-1840223, NIH R01 DK116187-01,
HFSP RGP0051/2017, NSF DMS 17-13012, and NSF DMS 16-13261.}

\author{Andrew Nobel}
\address{UNC Chapel Hill}
\email{nobel@email.unc.edu}
\thanks{Andrew Nobel would like to acknowledge funding from NSF DMS-1613261, NSF DMS-1613072, NIH R01 HG009125-01.}

\maketitle

\begin{abstract}
In this paper we consider a Bayesian framework for making inferences about dynamical systems from  ergodic observations. The proposed Bayesian procedure is based on the Gibbs posterior, a decision theoretic generalization of standard Bayesian inference. We place a prior over a model class consisting of a parametrized family of Gibbs measures on a mixing shift of finite type. This model class generalizes (hidden) Markov chain models by allowing for long range dependencies, including Markov chains of arbitrarily large orders. We characterize the asymptotic behavior of the Gibbs posterior distribution on the parameter space as the number of observations tends to infinity. In particular, we define a limiting variational problem over the space of joinings of the model system with the observed system, and we show that the Gibbs posterior distributions concentrate around the solution set of this variational problem. In the case of properly specified models our convergence results may be used to establish posterior consistency. This work establishes tight connections between Gibbs posterior inference and the thermodynamic formalism, which may inspire new proof techniques in the study of Bayesian posterior consistency for dependent processes.
\end{abstract}

\section{Introduction} \label{Sect:Intro}

In this work we establish asymptotic results concerning Bayesian inference for certain dynamical systems. We consider a fairly general framework in which both the observations and the fitted models arise from dynamical systems. Our analysis brings together two distinct strands of research, both of which were originally inspired by connections with statistical physics: the thermodynamic formalism in dynamical systems, and the Gibbs posterior principle for Bayesian inference in statistics. Our work highlights the substantial connections between these two areas and shows that together they produce a natural framework for Bayesian inference about dynamical systems. Our general results guarantee the concentration of Gibbs posterior distributions around certain sets of parameters that are characterized by a variational principle. As applications of these general results, we also establish posterior consistency results for some classes of dynamical models, generalizing previous posterior consistency results for Markov and hidden Markov models in Bayesian nonparametrics. 

\subsection{Observed system}

Our inference framework consists of two main components.  The first component is an observed dynamical system, defined as follows. Let $\mathcal{Y}$ be a complete separable metric space. Here and throughout this work we assume that all such spaces are endowed with their Borel $\sigma$-algebras, and we suppress this choice in our notation. Let $T : \mathcal{Y} \to \mathcal{Y}$ be a Borel measurable map. We let $\mathcal{M}(\mathcal{Y})$ denote the set of Borel probability measures on $\mathcal{Y}$, endowed with the weak$^*$ topology on measures. For $\nu \in \mathcal{M}(\mathcal{Y})$, we say that $\nu$ is invariant under $T$ if $\nu(T^{-1} E) = \nu(E)$ for all Borel sets $E \subset \mathcal{Y}$. The set of $T$-invariant measures in $\mathcal{M}(\mathcal{Y})$ is denoted by $\mathcal{M}(\mathcal{Y},T)$. Furthermore, we say that $\nu \in \mathcal{M}(\mathcal{Y},T)$ is ergodic if $\nu(E) \in \{0,1\}$ for all Borel sets $E$ satisfying $T^{-1}(E) = E$. Our standard assumption is that the observed system has the form $(\mathcal{Y},T,\nu)$, where $\nu \in \mathcal{M}(\mathcal{Y},T)$ is ergodic.

\subsection{Model families}
The second component of our inference framework is a collection of models. In order to model dynamics in the standard statistical setting, one typically considers (hidden) Markov models or more complex state space models. In our analysis we would like to be able to handle model processes with long range dependencies, and so we consider 
a general class of models known as Gibbs measures. This class of models strictly generalizes the class of finite state Markov models with arbitrarily large order.

Before giving a precise definition of a Gibbs measure, we must first introduce the underlying state space for such models, which is called a mixing shift of finite type (SFT). A shift of finite type is a dynamical system that is the topological analogue of a finite state aperiodic and irreducible Markov chain. SFTs have been widely studied in the dynamical systems literature, both for their own sake \cite{LindMarcus} and as model systems for some smooth systems such as Axiom A diffeomorphisms \cite{Bowen1975}. Furthermore, SFTs have substantial connections to statistical physics and other fields such as coding and information theory \cite{LindMarcus,Ruelle2004}. 

Here we give a proper definition for a mixing SFT. Let $\mathcal{A}$ be a finite set, known as an alphabet, and let $\Sigma = \mathcal{A}^{\mathbb{Z}}$ be the set of bi-infinite sequences $x = (x_n)$ with values in $\mathcal{A}$. Define the left-shift map $\sigma : \Sigma \to \Sigma$ by $\sigma(x)_{n+1} = x_n$. A set $\mathcal{X}$ is called an SFT if there exists $n \geq 0$ and a collection of words $\mathcal{W} \subset \mathcal{A}^n$ such that $\mathcal{X}$ is exactly the set of sequences in $\Sigma$ that contain no words from $\mathcal{W}$:
\begin{equation*}
\mathcal{X} = \bigl\{ x \in \Sigma : \forall i \in \mathbb{Z}, \, x_{i+1} \dots x_{i+n} \notin \mathcal{W} \}.
\end{equation*}
Here $\mathcal{W}$ is called a set of \textit{forbidden words} for $\mathcal{X}$. 
Note that by choosing $\mathcal{W} = \varnothing$, one obtains the full sequence space $\Sigma$, which is known as the full shift (on the alphabet $\mathcal{A}$).
Also, we endow $\mathcal{A}$ with the discrete topology and $\Sigma$ with the product topology, which makes any such $\mathcal{X}$ closed and compact.
We define the map $S : \mathcal{X} \to \mathcal{X}$ to be the restriction of the left shift $\sigma$ to $\mathcal{X}$. 
Let $\mathcal{L}_m$ denote the set of words of length $m$ (i.e., elements of $\mathcal{A}^m$) that appear in at least one point of $\mathcal{X}$, and let $\mathcal{L} = \cup_{m \geq 0} \mathcal{L}_m$.
An SFT $\mathcal{X}$ is said to be mixing if for any two words $u,v \in \mathcal{L}$, there exists $N$ such that for all $m \geq N$, there exists a word $w \in \mathcal{L}_m$ such that $uwv \in \mathcal{L}$. 
The following equivalent definition is perhaps more intuitive to readers familiar with Markov chains.
Let $A$ be the square matrix indexed by $\mathcal{A}^n$ defined for for words $u,v \in \mathcal{A}^n$ by the rule
\begin{equation*}
A_{uv} = \left\{ \begin{array}{ll}
                          1, & \text{if } \exists x \in \mathcal{X} \text{ such that } x_0^{n-1} = u \text{ and } x_1^n = v \\
                          0, & \text{otherwise}.
                         \end{array}
                         \right.
\end{equation*}
Then $\mathcal{X}$ is mixing if and only if there exists $N \geq 1$ such that $A^N$ contains all positive entries.
Our standard assumption on $\mathcal{X}$ is that it is a mixing SFT.

To model stochastic behavior on the topological system $(\mathcal{X},S)$, we consider a family of $S$-invariant probability measures on $\mathcal{X}$, called Gibbs measures.   
To introduce Gibbs measures, one begins with a function $f : \mathcal{X} \to \mathbb{R}$, which is called a potential function (or just a potential). 
A Borel probability measure $\mu$ on $\mathcal{X}$ is said to be a Gibbs measure corresponding to the potential function $f : \mathcal{X} \to \mathbb{R}$ if there exists constants $\mathcal{P} \in \mathbb{R}$ and $K>0$ such that for all $x \in \mathcal{X}$ and $m \geq 1$, 
\begin{equation} \label{Eqn:GibbsProp}
K^{-1} \leq \frac{\mu\bigl( x[0,m-1]) \bigr)}{\exp \Bigl( - \mathcal{P} m + \sum_{k=0}^{m-1} f \bigl( S^k (x) \bigr) \Bigr)} \leq K,
\end{equation}
where $x[0,m-1]$ is the cylinder set of points $y$ in $\mathcal{X}$ such that $x_i = y_i$ for all $i = 0,\dots,m-1$. The property in (\ref{Eqn:GibbsProp}) is called the Gibbs Property. By a celebrated result of Bowen \cite{Bowen1975}, under mild regularity conditions on $f$, there is a unique Gibbs measure $\mu \in \mathcal{M}(\mathcal{X},S)$ with potential function $f$, and furthermore the measure $\mu$ is ergodic. The constant $\mathcal{P} = \mathcal{P}(f)$ is called the \textit{pressure} 
of $f$.

The Gibbs measure is a generalization of the canonical ensemble in statistical physics to infinite systems. 
Potential functions have natural connections with Hamiltonians in the study of lattice systems in statistical physics. 
In considering inference, we will think of loss functions as potential functions.
We remark (again) that the class of Gibbs measures strictly generalizes the class of Markov chains, allowing for arbitrarily long dependencies. 
Indeed, any Markov chain of order $k$ on the alphabet $\mathcal{A}$ can be realized as a Gibbs measure by an appropriate choice of a potential function that depends on only $k$ coordinates.
On the other hand, when the potential function $f$ depends on infinitely many coordinates, the corresponding Gibbs measure is not Markov of any order.
In this way, our model families may include Markov chains with unbounded orders, which highlights the degree of dependence allowed by our framework.

Lastly, let us mention the regularity condition that we require our model families to satisfy. For points $x,y$ in $\mathcal{X}$, we let $n(x,y)$ denote the infimum of all $|m|$ such that $x_m \neq y_m$. Then we define a metric $d( \cdot, \cdot)$ on $\mathcal{X}$ by setting $d_{\mathcal{X}}(x,y) = 2^{-n(x,y)}$. For $r >0$, we let $C^{r}(\mathcal{X})$ denote the set of continuous functions from $\mathcal{X}$ to $\mathbb{R}$ with H\"{o}lder exponent $r$, that is, the set of functions $f : \mathcal{X} \to \mathbb{R}$ for which there exists a constant $c$ such that for all $x,y \in \mathcal{X}$, 
\begin{equation*}
|f(x) - f(y)| \leq c \, d_{\mathcal{X}}(x,y)^r.
\end{equation*} 
Furthermore, we endow $C^r(\mathcal{X})$ with the topology induced by the norm $\| \cdot \|_r$, where
\begin{equation*}
\| f \|_r = \sup_{x \in \mathcal{X}} | f(x) | + \sup_{x \neq y} \frac{ |f(x) - f(y)| }{ d_{\mathcal{X}}(x,y)^r }.
\end{equation*}
Now we define the regularity condition necessary for our model families.

\begin{definition} \label{Def:RegularFamily}
Let $\Theta$ be compact metric space. A parametrized family of potential functions $\mathcal{F} = \{ f_{\theta} : \theta \in \Theta\}$ will be called a \textit{regular family} if there exists $r>0$ such that $\mathcal{F} \subset C^{r}(\mathcal{X})$ and the map $\theta \to f_{\theta}$ is continuous in the topology induced by the norm $\| \cdot \|_r$.
\end{definition}

If a family $\{ f_{\theta} : \theta \in \Theta\}$ is a regular family, then the map $\theta \mapsto \mu_{\theta}$ is continuous in the weak$^*$ topology on measures, and the constants $K(f_{\theta})$ and $\mathcal{P}(f_{\theta})$ that appear in (\ref{Eqn:GibbsProp}) depend continuously on $\theta$ (see \cite{Alves2018}). Furthermore, since $\Theta$ is compact, we get a uniform Gibbs property: there exists a uniform constant $K$ and a continuous function $\theta \mapsto \mathcal{P}(f_{\theta})$ such that for all $\theta \in \Theta$, $x \in \mathcal{X}$, and $m \geq 1$, 
\begin{equation} \label{Eqn:UniformGibbs}
K^{-1} \leq \frac{\mu_{\theta}( x[0,m-1] ) }{ \exp\Bigl( - \mathcal{P}(f_{\theta}) m + \sum_{k = 0}^{m-1} f_{\theta}( S^k x) \Bigr)} \leq K.
\end{equation}

We assume throughout that $\mathcal{F} = \{ f_{\theta} : \theta \in \Theta\}$ is a regular family of potential functions, and our model class consists of the corresponding parametrized family of Gibbs measures $\{ \mu_{\theta} : \theta \in \Theta\}$.

\subsection{Inference} \label{Sect:Inference}

The inference paradigm we consider is known as Gibbs posterior inference, which is a generalization of the standard Bayesian inference framework. 
The basic idea behind the Gibbs posterior   \cite{Bissiri2016,Jiang2008} is to replace the likelihood with an exponentiated loss or utility function in the standard Bayesian procedure for updating beliefs about an unknown parameter of interest $\theta$. 
Whereas the standard Bayes posterior takes the form
\begin{eqnarray*}
 \pi(\theta \mid \mbox{data}) &=& \frac{\mbox{Likelihood}(\mbox{data} \mid \theta) \times \pi(\theta)}{\int_\Theta  \mbox{Likelihood}(\mbox{data} \mid \theta) \times \pi(\theta) \, d\theta },
 \end{eqnarray*}
 the Gibbs posterior has the form
 \begin{eqnarray*}
  \pi(\theta \mid \mbox{data}) &=& \frac{ \exp(-\ell(\mbox{data},\theta)) \times \pi(\theta)}{\int_\Theta \exp(-\ell(\mbox{data},\theta)) \times \pi(\theta) \, d\theta },
  \end{eqnarray*}
  where $\ell( \mbox{data}, \theta)$ is the loss associated with $\theta$ based on the observed data. 
When the loss function is the negative log-likelihood then the two paradigms are identical. The original motivation for the Gibbs posterior was to specify
a  coherent procedure for Bayesian inference when the parameter of interest is connected to observations via a loss function, rather than the classical setting where the likelihood or true sampling distribution is known; see \cite{Bissiri2016} for more arguments in favor of the Gibbs posterior and discussion about how the Gibbs posterior framework addresses model 
misspecification and robustness to nuisance parameters. Note that in the general Gibbs posterior framework without a likelihood, there is no generative model assumed for the observations. 

We consider models indexed by  a compact metric space $\Theta$ (with metric denoted $d_{\Theta}$), which will serve as a parameter space. 
The elements of $\Theta$ will be used to parametrize both the dependence structure of the Gibbs measures in our model class (e.g., transition probabilities) and the relationship between states and observations (e.g., emission probabilities). 
Recall that as part of our standard assumptions, we assume that our model class $\{\mu_{\theta} : \theta \in \Theta\}$ is a family of Gibbs measures on $\mathcal{X}$ corresponding to a regular family of potential functions (as in Definition \ref{Def:RegularFamily}). 

Also recall that the observed system has state space $\mathcal{Y}$ with invariant measure $\nu$.
Define a metric on $\Theta \times \mathcal{X}$ by the rule 
\begin{equation*}
d \bigl( (\theta,x), (\theta',x') \bigr) = \max \bigl( d_{\Theta}(\theta,\theta'), d_{\mathcal{X}}(x,x') \bigr).
\end{equation*}
Here and throughout this work, we assume that we have a loss function $\ell : \Theta \times \mathcal{X} \times \mathcal{Y} \to \R$ satisfying the following conditions:
\renewcommand{\theenumi}{\roman{enumi}}
\renewcommand{\labelenumi}{(\theenumi)}
\begin{enumerate}
\vspace{1mm}
\item $\ell$ is continuous;
\vspace{1mm}
\item there exists a measurable function $\ell^* : \cY \to \R$ such that for all $y \in \cY$, $\sup_{\theta,x} |\ell(\theta,x,y)| \leq \ell^*(y)$, and $\int \ell^* \, d\nu < \infty$;
\vspace{1mm}
\item for each $\delta>0$ there exists a measurable function $\rho_{\delta} : \cY \to (0,\infty)$ such that for each $y \in \cY$,
\begin{equation*}
\quad \quad \sup \bigl\{ |\ell(\theta,x,y) - \ell(\theta',x',y)| :  \, d \bigl((\theta,x), (\theta',x') \bigr) \leq \delta \bigl\} \leq \rho_{\delta}(y),
\end{equation*}
and $\lim_{\delta \to 0^+} \int \rho_{\delta} \, d\nu = 0$.
\end{enumerate}
Condition (ii) is an integrability condition on the loss, while condition (iii) is a requirement on the modulus of continuity of the loss.
In Section \ref{Sect:ExampleLossFunctions} we provide examples of loss functions satisfying these conditions. Note that including the  parameter $\theta$ in the loss function may be considered non-standard in statistics. However, this formulation will simplify notation throughout the paper, and in Section \ref{Sect:ExampleLossFunctions} we establish that this setting is equivalent to the standard one. Also note that the dependence of the loss on $\Theta$ and on the uncountable space $\mathcal{X}$ allows us to model continuous observations and emission probabilities.

With the loss function and parameter $\theta \in \Theta$ fixed, we define the loss of the finite sequence $x_0^{n-1} \in \cX^n$ with respect to a finite sequence of observations $y_0^{n-1} \in \cY^n$ to be the sum of the per-state losses:
\begin{equation} \label{Eqn:Ellen}
\ell_n(\theta; x_0^{n-1},y_0^{n-1}) = \sum_{k=0}^{n-1} \ell( \theta, x_k, y_k).
\end{equation}
When $x_0^{n-1} = (x,Sx, \dots, S^{n-1}x)$ and $y_0^{n-1} = (y,Ty,\dots,T^{n-1}y)$ are initial segments of trajectories of $S$ and $T$, respectively, we write $\ell_n(\theta,x,y)$ instead of $\ell_n(\theta; x_0^{n-1},y_0^{n-1})$.

Let us now give the definition of Gibbs posterior distributions on $\Theta$. Here we consider the subjective case, in which one begins with a prior distribution on $\Theta$. 
Let $\pi_0$ be a fully supported Borel probability measure on $\Theta$, which will serve as our prior distribution. First we extend $\pi_0$ to form a prior distribution on $\Theta \times \mathcal{X}$ as follows. 
Given the family $\{ \mu_{\theta} : \theta \in \Theta\}$ of Gibbs measures on $\mathcal{X}$, consider the induced prior distribution $P_0$ on $\Theta \times \cX$ defined for any Borel set $E \subset \Theta \times \cX$ by
\begin{equation} \label{Eqn:BigPrior}
P_0(E) = \int \int \mathbf{1}_{E}(\theta,x) \, d\mu_{\theta}(x) \, d\pi_0(\theta).
\end{equation}
According to the Gibbs posterior paradigm \cite{Bissiri2016,Jiang2008}, if we make observations $(y,Ty,\dots,T^{n-1}y) \in \cY^n$, then our updated beliefs should be represented by the Gibbs posterior distribution. This distribution is the Borel probability measure $P_n( \cdot \mid y)$ on $\Theta \times \cX$ defined for Borel sets $E \subset \Theta \times \cX$ by
\begin{align} \label{Eqn:GibbsDist}
P_n(E \mid y) & = \frac{ 1}{ Z_n(y) } \int_E \exp \bigl( -\ell_n(\theta,x,y) \bigr) \, dP_0(\theta,x) \\
 & = \frac{ 1}{ Z_n(y) } \int \int \mathbf{1}_E(\theta,x) \exp \bigl( -\ell_n(\theta,x,y) \bigr) \, d\mu_{\theta}(x) \, d\pi_0(\theta),
\end{align}
where $Z_n(y)$ is the normalizing constant (partition function), given by
\begin{equation*}
Z_n(y) = \int \exp \bigl( -\ell_n(\theta,x,y) \bigr) \, dP_0(\theta,x).
\end{equation*}
Then the Gibbs posterior distribution $\pi_n( \cdot \mid y)$ on $\Theta$ is simply the $\Theta$-marginal of $P_n( \cdot \mid y)$, which is defined for Borel sets $E \subset \Theta$ by
\begin{equation*}
\pi_n( E \mid y) = P_n( E \times \mathcal{X} \mid y).
\end{equation*}

As we are considering a Bayesian framework, all inference about the parameters $\theta \in \Theta$ based on the observations $(y, Ty, \dots, T^{n-1}y)$ is derived from the posterior $\pi_n( \cdot \mid y)$. 
We focus here on inference regarding the parameters in $\Theta$, since inference regarding the initial condition $x$ in $\mathcal{X}$ is known to be impossible for many dynamical systems, including shifts of finite type \cite{Lalley1999,LalleyNobel2006}.
Let us summarize our framework.
\begin{itemize}
\item We begin with a fully supported prior $\pi_0$ on a compact set $\Theta$ that smoothly parametrizes a family of Gibbs measures $\{\mu_{\theta} : \theta \in \Theta\}$ on $\mathcal{X}$.
\item From $\pi_0$ and $\{\mu_{\theta} : \theta \in \Theta\}$, we create an extended prior $P_0$ on $\Theta \times \mathcal{X}$.
\item We obtain observations $y,\dots, T^{n-1}y$ in $\mathcal{Y}$ from a stationary ergodic process $(\mathcal{Y},T,\nu)$.
\item From $P_0$, the observations, and the loss function $\ell$ we obtain the Gibbs posterior $P_n$ on $\Theta \times \mathcal{X}$.
\item Finally, we marginalize $P_n$ to get the posterior $\pi_n$ on $\Theta$.
\end{itemize}

\subsection{Main results}

Our analysis begins with an examination of the exponential growth rate of the (random) partition function $Z_n$ for large $n$. In particular, we establish a variational principle for the almost sure limit of $n^{-1} \log Z_n$ as $n$ tends to infinity.

\begin{theorem} \label{Thm:Pressure} 
Under the standard assumptions stated above there exists a lower semicontinuous function $V : \Theta \to \R$ such that for $\nu$-almost every $y \in \mathcal{Y}$,
\begin{equation*}
\lim_n -\frac{1}{n} \log Z_n(y) \ = \ \inf_{\theta \in \Theta} V(\theta).
\end{equation*}
\end{theorem}

\begin{remark} \label{Rmk:V}
The compactness of $\Theta$ and lower semicontinuity of $V$ ensure that the infimum in Theorem \ref{Thm:Pressure} is obtained.
The conclusion of the theorem is similar to a large deviations principle (see for example \cite{Dupuis2011}), with $V : \Theta \to \R$ playing the role of the rate function. For this reason, we refer to $V$ as the rate function in this setting. 
A detailed discussion of $V$ appears in Section \ref{Sect:IntroRateFunction}, where we show that $V$ can be expressed
as the sum of an expected loss term and a divergence term. 
\end{remark}

The variational expression that appears in Theorem \ref{Thm:Pressure} suggests that we focus on the (non-empty, compact) set of parameters $\theta$ that minimize this expression. Let
\begin{equation*}
\Theta_{\min} =  \argmin_{\theta \in \Theta} V(\theta).
\end{equation*}  
In our second main result, we establish that the Gibbs posterior distribution must concentrate around this set.

\begin{theorem} \label{Thm:PosteriorConvergence}
For any open neighborhood $U$ of $\Theta_{\min}$, for $\nu$-almost every $y \in \mathcal{Y}$, we have
\begin{equation*}
\lim_n \pi_n \bigl(\Theta \setminus U \mid y \bigr) = 0.
\end{equation*}
\end{theorem}
In light of this result, it is possible to answer questions about Gibbs posterior consistency by analyzing the variational problem defining $\Theta_{\min}$. We illustrate this approach to posterior consistency in several applications (see Section \ref{Sect:Examples}).

\begin{remark}[Optimality of $\Theta_{\min}$]
One may wonder whether $\pi_n( \cdot \mid y)$ actually concentrates around a strict subset of $\Theta_{\min}$. 
Proposition \ref{Prop:TakeABite} addresses this question on the exponential scale. It states that if $U \subset \Theta$ is open and intersects $\Theta_{\min}$, then the posterior probability of $U$ cannot be exponentially small as $n$ tends to infinity, i.e., for $\nu$-almost every $y$, the quantity $n^{-1} \log \pi_n(U \mid y)$ tends to zero as $n$ tends to infinity. 
\end{remark}

\begin{remark}[Ground states and MAP] \label{Rmk:GroundStates} 
From a thermodynamic perspective, it is natural to introduce an inverse temperature parameter $\beta \in \R$ and consider the new loss function $\ell_\beta(\theta,x,y) = \beta \cdot \ell(\theta,x,y)$. In this setting, one would like to understand what happens as $\beta$ tends to infinity. 
In Section \ref{Sect:RateFunction}, we identify the limit of both $V$ and $\Theta_{\min}$ as $\beta$ tends to infinity in terms of variational problems considered in previous work \cite{McGoff2016variational}.

The use of an inverse temperature parameter has also been used in practice to perform maximum a posteriori (MAP) estimation.
MAP estimation is a common alternative to fully Bayesian inference that is used in both statistics and machine learning. It involves finding the parameter that is the posterior mode. The motivation for MAP estimation is often computational efficiency and the lack of a need for uncertainty quantification. The idea of adding an inverse temperature parameter ($\beta$)  to a Gibbs distribution for MAP estimation was introduced for Bayesian models in a seminal paper by Geman and Geman \cite{GemanGeman1984}, who also gave an annealing schedule to increase the inverse temperature with a provable guarantee for finding the posterior mode.
\end{remark}

\begin{remark}[Connections to penalization]
The formulation of Bayesian updating as a variational problem with an entropic penalty has been previously explored \cite{Bissiri2016,Zellner1988}, and these ideas are related to Jaynes' maximum entropy formulation of Bayesian inference \cite{Jaynes1968}. In both \cite{Bissiri2016} and \cite{Zellner1988}, posterior inference was formulated as follows:
given a loss function $\ell(\theta,x)$ and a prior $\pi$, the posterior distribution is
$$\pi(\theta \mid x) = \arg \min_{\mu} \biggl\{ \int_\theta \ell(\theta,x) \, d\mu(\theta) + d_{KL}(\mu, \pi) \biggr\},$$ 
where $d_{KL}(\mu, \pi)$ is the relative entropy between $\mu$ and $\pi$. The function being minimized above has close connections to the function $V(\theta)$ in Theorem \ref{Thm:Pressure}; see Definition \ref{Defn:V} below.
\end{remark}

\begin{remark}[Convergence of full Gibbs posteriors]
Our main results establish the concentration of the $\Theta$-marginal posterior distributions $\pi_n( \cdot \mid y)$ around the limit set $\Theta_{\min}$. In contrast, the $\mathcal{X}$-marginal of the full posterior distribution $P_n( \cdot \mid y)$ need not concentrate around any particular subset of $\mathcal{X}$ (according to the negative results of \cite{Lalley1999,LalleyNobel2006}). Nonetheless, Proposition \ref{Prop:CesaroConvergence} gives a characterization of any Ces\`{a}ro limit of the full posteriors.
\end{remark}

\begin{remark}[Importance of the Gibbs property]
The Gibbs property (\ref{Eqn:GibbsProp}) of the measures $\mu_{\theta}$ makes them particularly suitable as model distributions for the purposes of Gibbs posterior inference. In general, invariant measures for dynamical systems do not admit such exponential estimates, and it is precisely these estimates that allow us to carry out our asymptotic analysis.
\end{remark}

\begin{remark}[Continuous model systems]
Results similar to those here may be established for certain families of differential dynamical systems on manifolds. In particular, using our results and the well-known connections between SFTs and Axiom A systems (see \cite{Bowen1975}), it is possible to establish analogous conclusions for Axiom A diffeomorphisms with Gibbs measures. 
\end{remark}

\subsection{Examples of inference settings and associated loss functions} \label{Sect:ExampleLossFunctions}

Here we describe some possible inference settings that fit into our framework. Note that these settings give examples of loss functions that satisfy conditions (i)-(iii).

\begin{example}(Continuous, deterministic observations) \label{Ex:CtsObs}
Suppose that the state space $\mathcal{Y}$ of the observed system is a subset of the real line, so that 
the observations $y, Ty, T^2y, \ldots$ are real-valued and deterministic.  In this case, we might wish to
fit the observations to a family of models generated by a family $\{\mu_{\theta} : \theta \in \Theta\}$ of Gibbs 
measures on a fixed mixing SFT $\mathcal{X}$, with associated prior $\pi_0$, and a continuously parametrized family 
$\{\varphi_{\theta} : \mathcal{X} \to \R \}$ of continuous observation functions.  
Given $\theta$ and $x$, the initial part of the real-valued sequence $\{\varphi_{\theta}(S^kx)\}_{k \geq 0}$ can be fit to the 
the observations.  Models of this sort are called dynamical models, and they have been studied in the context of 
empirical risk minimization in \cite{McGoff2016empirical}. 
If the measure $\nu$ has finite second moment, and $\ell$ is the squared loss 
$\ell(\theta,x,y) = | \varphi_{\theta}(x) - y|^2$, then conditions (i)-(iii) on the loss are satisfied. 
\end{example}

\begin{example}(Discrete observations) \label{Ex:DiscreteObs}
Let $\mathcal{A}$ and $\mathcal{B}$ be finite sets. We suppose that we make $\mathcal{B}$-valued observations, that is, $\mathcal{Y} \subset \mathcal{B}^{\mathbb{Z}}$, and we wish to model these observations with a family of Gibbs measures $\{\mu_{\theta} : \theta \in \Theta\}$ on a mixing SFT $\mathcal{X}$ with $\mathcal{X} \subset \mathcal{A}^{\mathbb{Z}}$. Further, let $\varphi : \mathcal{A} \to \mathcal{B}$ be an observation function, so that a point $x$ in $\mathcal{X}$ gives rise to the $\mathcal{B}$-valued sequence $\{ \varphi( x_k) \}_{k \geq 0}$. Let $\ell$ be the discrete loss, $\ell(\theta,x,y) = \mathbf{1}( \varphi(x_0) \neq y_0 )$.  Then the conditions (i)-(iii) on the loss are satisfied.
\end{example}

\begin{example}(Family of conditional likelihoods) \label{Ex:ProperLikelihood}
Suppose that $\{p( \cdot \mid x, \theta) : \theta \in \Theta, x \in \mathcal{X}\}$ is a family of conditional densities on $\mathcal{Y}$ with respect to a common Borel measure $m$ on $\mathcal{Y}$. Here $p( \cdot \mid x, \theta)$ is the conditional likelihood of a single observation given the parameter $\theta$ and system state $x$. Under appropriate continuity and integrability conditions on the family of likelihoods, the negative log-likelihood function, $\ell(\theta,x,y) = - \log p(y \mid x, \theta)$, satisfies conditions (i)-(iii). In this situation, the Gibbs posterior is the same as the standard Bayes posterior. Furthermore, the dependence of the loss on the parameter $\theta$ allows one to parametrize the conditional observation densities, as in the parametrization of emission densities in the study of hidden Markov models. Note that in the Gibbs posterior framework, the true observation system $(\mathcal{Y},T,\nu)$ may be fully misspecified--it need not be related to any of the generative processes implied by the family of Gibbs measures and conditional likelihoods.
\end{example}

%

\section{Applications} \label{Sect:Examples}

In this section we present two applications of our main results on Gibbs posterior consistency to standard posterior consistency for two models.
In the first, we establish Bayesian posterior consistency for direct observations of Gibbs processes. Interestingly, the proof of this result may be reduced to a classical result of Bowen on uniqueness of equilibrium states. In the second application, we establish Bayesian posterior consistency for hidden Gibbs processes. This result generalizes previous results on posterior consistency for hidden Markov models by allowing substantially more dependence in the hidden processes, including families of Markov chains with unbounded orders.

\subsection{Direct observations of Gibbs processes} \label{Sect:DirectGibbs}

Let $\mathcal{Z}$ be a mixing SFT, and let $\{ f_{\theta} : \theta \in \Theta\} \subset C^r(\mathcal{Z})$ be a regular family of H\"{o}lder potential functions (as in Definition \ref{Def:RegularFamily}). We let $\mu_{\theta}$ be the unique Gibbs measure associated with $f_{\theta}$. 
Additionally, let $\Pi_0$ be a fully supported prior distribution on $\Theta$. 

Here we consider the properly specified situation with direct observations. That is, we suppose that there exists $\theta^* \in \Theta$ such that our observations $Y_0,Y_1,\dots$ are the coordinates of the ergodic system $(\mathcal{Z},\sigma|_{\mathcal{Z}},\mu_{\theta^*})$, observed without noise. 
In other words, the observation $Y_k$ at time $k \geq 0$ is the $k$th coordinate of the stationary ergodic process 
$\mathbf{Y} = \{Y_k\}_{k \in \mathbb{Z}} \in \mathcal{Z}$ with distribution $\mu_{\theta^*}$. 
In this case, the likelihood of observing $y_0 \dots y_{n-1}$ under $\theta$ is simply given by $\mu_{\theta}([y_0^{n-1}])$, where $[y_0^{n-1}] = \{ z \in \mathcal{Z} : z_0^{n-1} = y_0^{n-1} \}$.  
Let $\Pi_n( \cdot \mid y_0^{n-1})$ be the standard Bayesian posterior distribution on $\Theta$ given the observations $y_0^{n-1}$, 
i.e., for Borel sets $A \subset \Theta$,
\begin{equation*}
 \Pi_n(A \mid y_0^{n-1}) = \frac{ \int_A \mu_{\theta}([y_0^{n-1}]) \, d\Pi_0(\theta) }{ \int_{\Theta} \mu_{\theta}([y_0^{n-1}]) \, d\Pi_0(\theta) }.
\end{equation*}
Let $[\theta^*]$ denote the identifiability class of $\theta^*$, which is naturally defined as the set of all $\theta$ such that $\mu_{\theta} = \mu_{\theta^*}$. 
The following theorem states that $\Pi_n( \cdot \mid Y_0^{n-1})$ concentrates around $[\theta^*]$, establishing posterior consistency in this setting. 
\begin{theorem} \label{Thm:DirectGibbs}
Let $U \subset \Theta$ be an open neighborhood of $[\theta^*]$. Then with probability one (with respect to $\mu_{\theta^*}$), we have
\begin{equation*}
\lim_n \Pi_n \bigl( \Theta \setminus U \mid Y_0^{n-1} \bigr) = 0.
\end{equation*}
\end{theorem}

The proof of Theorem \ref{Thm:DirectGibbs} is based on the principal results above. 
In particular, we apply Theorem \ref{Thm:PosteriorConvergence} to establish that the posterior concentrates around a set $\Theta_{\min}$, which is characterized as the solution set of a variational problem. 
We then use the variational characterization of $\Theta_{\min}$ and additional problem-specific arguments to show that $\Theta_{\min}$ is equal to the identifiability class $[\theta^*]$. 
In this example, these problem-specific arguments rely on one of the foundational results of Bowen \cite{Bowen1975} 
in the thermodynamic formalism, namely the uniqueness of equilibrium states for H\"{o}lder continuous 
potentials on a mixing SFT. 

\subsection{Hidden Gibbs processes} \label{Sect:HiddenGibbs}

In this section we consider posterior consistency for more general observation processes. Let $\mathcal{X}$ be a mixing SFT, $\{ f_{\theta} : \theta \in \Theta\} \subset C^r(\mathcal{X})$  a regular family of H\"{o}lder potential functions (as in Definition \ref{Def:RegularFamily}), and $\{ \mu_{\theta} : \theta \in \Theta\}$  the corresponding family of Gibbs measures. Let $\Pi_0$ be any fully supported prior distribution on $\Theta$.

The novel feature of the present setting is that we allow for general observations of the underlying family. Suppose that $m$ is a $\sigma$-finite Borel measure on a complete separable metric space $\mathcal{U}$, and that $\varphi : \Theta \times \mathcal{X} \times \mathcal{U} \to [0,\infty)$ is a jointly continuous function such that for all $\theta \in \Theta$ and $x \in \mathcal{X}$,
\begin{equation*}
\int \varphi_{\theta}( u \mid x) \, dm(u) = 1.
\end{equation*}
We regard $\{ \varphi_{\theta}( \cdot \mid x) : \theta \in \Theta, \, x \in \mathcal{X} \}$ as a family of conditional likelihoods for $u \in \mathcal{U}$ given $\theta$ and $x$.
We assume that the function $L : \Theta \times \mathcal{X} \times \mathcal{U} \to \R$ given by $L(\theta,x,u) = - \log \varphi_{\theta}(u \mid x)$ satisfies the integrability and regularity conditions (i)-(iii) from Section \ref{Sect:Inference}. Furthermore, we require condition (L2) from \cite{McGoff2015}, which stipulates that there exists $\alpha >0$ and a Borel measurable function $C : \Theta \times \mathcal{U} \to [0,\infty)$ such that for each $(\theta,u) \in \Theta \times \mathcal{U}$, the function $L(\theta, \cdot, u) : \mathcal{X} \to \R$ is $\alpha$-H\"{o}lder continuous with constant $C(\theta,u)$, and for each $\beta >0$,
\begin{equation*}
\sup_{(\theta,x) \in \Theta \times \mathcal{X}} \int \exp\bigl( \beta C(\theta,u) \bigr) \varphi_{\theta}(u \mid x) \, dm(u) < \infty.
\end{equation*}
This condition may be viewed as a condition on the regularity of the conditional density functions; it is used in \cite{McGoff2015} to control the likelihood function in the large deviations regime.

With these conditions in place, we assume that the conditional likelihood of observing $u_0^{n-1} \in \mathcal{U}^{n}$ given $(\theta,x) \in \Theta \times \mathcal{X}$ is
\begin{equation*}
p_{\theta}\bigl(u_0^{n-1} \mid x\bigr) = \prod_{k=0}^{n-1} \varphi_{\theta}\bigl(u_k \mid S^kx \bigr),
\end{equation*}
and that the likelihood of observing $u_0^{n-1} \in \mathcal{U}^n$ given $\theta \in \Theta$ is
\begin{equation*}
p_{\theta}\bigl(u_0^{n-1}\bigr) = \int p_{\theta}\bigl( u_0^{n-1} \mid x\bigr) \, d\mu_{\theta}(x).
\end{equation*}
In other words, for each $\theta \in \Theta$, we have an observed sequence $U_0, U_1, \ldots$ generated as follows: select
$X \in \mathcal{X}$ according to $\mu_{\theta}$ and for each $k \geq 0$ let $U_k \in \mathcal{U}$ have density 
$\varphi_{\theta}( \cdot \mid S^k X)$ with respect to $m$. 
Denote by $\mathbb{P}^U_{\theta}$ the process measure for the process $\{U_k\}$, which has likelihood $p_{\theta}$.

Now let $\Pi_n( \cdot \mid u_0^{n-1})$ be the standard Bayesian posterior distribution on $\Theta$ 
given observations $u_0^{n-1}$ based on the prior $\Pi_0$ and the likelihood $p_{\theta}$:
for Borel sets $E \subset \Theta$,
\begin{equation*}
\Pi_n( E \mid u_0^{n-1} ) = \frac{\int_{E} p_{\theta}(u_0^{n-1}) \, d\Pi_0(\theta)}{ \int_{\Theta} p_{\theta}(u_0^{n-1}) \, d\Pi_0(\theta)}.
\end{equation*}

We again consider the properly specified case, in which there exists a parameter 
$\theta^* \in \Theta$ such that the observed process $\{Y_n\}$ is drawn from $\mathbb{P}_{\theta^*}^{U}$. 
In order to address posterior consistency, we define the identifiability class of $\theta^*$, denoted $[\theta^*]$, to 
be the set of $\theta \in \Theta$ such that $\mathbb{P}^U_{\theta} = \mathbb{P}^U_{\theta^*}$; in other words, 
a parameter is in $[\theta^*]$ if its associated process has the same distribution as the process 
generated by $\theta^*$. The following result establishes posterior consistency in this setting.

\begin{theorem} \label{Thm:HiddenGibbs}
Let $E \subset \Theta$ be an open neighborhood of $[\theta^*]$. Then 
\begin{equation*}
\lim_n \Pi_n \bigl( \Theta \setminus E \mid Y_0^{n-1} \bigr) = 0, \quad \quad \mathbb{P}^U_{\theta^*}-\text{a.s.}
\end{equation*}
\end{theorem}

The proof of Theorem \ref{Thm:HiddenGibbs} is based on the principal results above. 
In particular, we use these results to establish convergence of the posterior distribution, 
and then we use problem specific arguments to prove that the limit set $\Theta_{\min}$ 
is equal to the identifiability set $[\theta^*]$. In this case, the problem-specific arguments rely 
on previously studied connections between large deviations for Gibbs measures and identifiability 
of observed systems \cite{McGoff2015}.

\section{Joinings, divergence, and the rate function} \label{Sect:IntroRateFunction}

In this section we discuss the rate function $V : \Theta \to \R$, whose existence is asserted by Theorem \ref{Thm:Pressure}. In order to provide a thorough discussion, we first recall some background material from ergodic theory, including joinings and fiber entropy.




\subsection{Joinings} \label{Sect:Joinings}

Joinings were introduced by Furstenberg \cite{Furstenberg1967}, and they have played an important role in the development of ergodic theory (see \cite{Rue2006,Glasner2003}).
Suppose $(\mathcal{U}_0,R_0,\eta_0)$ and $(\mathcal{U}_1,R_1,\eta_1)$ are two probability measure-preserving Borel systems with $R_i : \mathcal{U}_i \to \mathcal{U}_i$ and $\eta_i \in \mathcal{M}(\mathcal{U}_i,R_i)$. The product transformation $R_0 \times R_1 : \mathcal{U}_0 \times \mathcal{U}_1 \to \mathcal{U}_0 \times \mathcal{U}_1$ is defined by $(R_0 \times R_1)(u,v) = (R_0(u), R_1(v))$. A \textit{joining} of these two systems is a Borel probability measure $\lambda$ on $\mathcal{U}_0 \times \mathcal{U}_1$ with marginal distributions $\eta_0$ and $\eta_1$ that is invariant under the product transformation $R_0 \times R_1$. 
Thus, a joining is a coupling of the measures $\eta_0$ and $\eta_1$ that is also invariant under the joint action of the transformations
$R_0$ and $R_1$; the former condition concerns the invariant measures of the two systems, while the latter concerns
their dynamics.  
Let $\mathcal{J}(\eta_0,\eta_1)$ denote the set of all joinings of $(\mathcal{U}_0,R_0,\eta_0)$ and $(\mathcal{U}_1,R_1,\eta_1)$. Note that this set is non-empty, since the product measure $\eta_0 \otimes \eta_1$ is always a joining.  
When the transformation $R_0 : \mathcal{U}_0 \to \mathcal{U}_0$ is fixed but we have not associated any invariant measure with it, we set
\begin{equation*}
\mathcal{J}( R_0 : \eta_1) = \bigcup_{\eta_0 \in \mathcal{M}(\mathcal{U}_0,R_0)} \mathcal{J}( \eta_0, \eta_1),
\end{equation*}
which is the family of joinings of $(\mathcal{U}_1,R_1,\eta_1)$ with all systems of the form $(\mathcal{U}_0,R_0,\eta_0)$, with $\eta_0 \in \mathcal{M}(\mathcal{U}_0,R_0)$.

\subsection{Entropy} \label{Sect:Entropy}

Our statements and proofs also require us to introduce some notions from the entropy theory of dynamical systems.
Let $\mathcal{U}$ be a compact metric space, $R : \mathcal{U} \to \mathcal{U}$ continuous, and $\eta \in \mathcal{M}(\mathcal{U},R)$.
For any finite measurable partition $\alpha$ of $\mathcal{U}$, we define
\begin{equation*}
H(\eta,\alpha) = - \sum_{C \in \xi} \eta(C) \log \eta(C),
\end{equation*}
where $0 \cdot \log 0 = 0$ by convention. For $k \geq 0$, let $R^{-k} \alpha = \{ R^{-k} A : A \in \alpha\}$, and for any partitions $\alpha^{0},\dots,\alpha^n$, define their join to be the mutual refinement
\begin{equation*}
\bigvee_{k=0}^{n} \alpha^k = \bigl\{ A_0 \cap \dots \cap A_n : A_i \in \alpha^i \}.
\end{equation*} 
For $n \geq 0$, let $\alpha_n = \bigvee_{k=0}^{n-1} R^{-k} \alpha$. 
By standard subadditivity arguments, the following limit exists:
\begin{equation*}
h_R(\eta,\alpha) := \lim_n \frac{1}{n} H(\eta, \alpha_n) = \inf_n \frac{1}{n} H(\eta, \alpha_n).
\end{equation*}
The measure-theoretic or Kolmogorov-Sinai entropy of $(\mathcal{U},R)$ with respect to $\eta$ is given by $h_R(\eta) = \sup_{\alpha} h_R(\eta,\alpha)$, where the supremum is taken over all finite measurable partitions $\alpha$ of $\mathcal{U}$.
We note for future reference that for any $\epsilon>0$, the value $h_R(\eta)$ remains the same if the supremum is instead taken over all finite measurable partitions with diameter less than $\epsilon$.
When the transformation $R$ is clear from context, we may omit the subscript.

\subsection{The variational principle for pressure} \label{Sect:VP}

Let $\mathcal{X}$ be a mixing SFT, and let $f : \mathcal{X} \to \mathbb{R}$ be a H\"{o}lder continuous potential. The variational principle  \cite{Bowen1975} for the pressure $\mathcal{P}(f)$ states that
\begin{equation}
\label{Eqn:VPTP}
 \mathcal{P}(f) = \sup \biggl\{ \int f \, d\mu + h(\mu) : \mu \in \mathcal{M}(\mathcal{X},S) \biggr\},
\end{equation}
and furthermore, the supremum is achieved by the measure $\mu \in \mathcal{M}(\mathcal{X},S)$ if and only if $\mu$ is the Gibbs measures associated with $f$.

\subsection{Disintegration of measure} \label{Sect:DisintegrationOfMeasure}
The following result is a special case of standard results on disintegration of Borel measures (see \cite{Glasner2003}).

\begin{theorem*}[Disintegration of measure] \label{Thm:Disintegration}
Let $\cU$ and $\cY$ be standard Borel spaces, and $\phi : \cU \times \cY \to \cY$ be the natural projection. Let $\lambda \in \cM(\cU \times \cY)$, and let $\nu = \lambda \circ \phi^{-1}$ be its image in $\cM(\cY)$. Then there is a Borel map $y \mapsto \lambda_y$, from $\cY$ to $\cM(\cU)$ such that
for every bounded Borel function $f: \cU \times \cY \to \R$,
\begin{equation*}
\int f d\lambda = \int \biggl( \int f \, d[\lambda_y \otimes \delta_y] \biggr) \, d\nu(y).
\end{equation*}
Moreover, such a map is unique in the following sense: if $y \mapsto \lambda'_y$ is another such map, then $\lambda_y = \lambda'_y$ for $\nu$-almost every $y$.
\end{theorem*}

Note that if $\lambda$ is a joining, then the family $\{\lambda_y\}_{y \in \mathcal{Y}}$ satisfies an important invariance property, which we state as Lemma \ref{Lemma:InvarianceOfDisintegration} in Section \ref{Sect:DisintegrationResults}.

\subsection{Fiber entropy} \label{Sect:FiberEntropy}

Now we give a definition of fiber entropy, along with statements of some properties relevant to this work;  for a thorough introduction, see \cite{Kifer2006}.
Let $\cU$ be a compact metric space and $\cY$ be a separable complete metric spaces. Further, let $R : \cU \to \cU$ be continuous and $T : \cY \to \cY$ be Borel measurable.
For any Borel probability measure $\lambda$ on $\mathcal{U} \times \cY$ with $\cY$-marginal $\nu$, let $\lambda = \int \lambda_y \otimes \delta_y \, d\nu(y)$ be its disintegration over $\cY$. Then for any finite measurable partition $\alpha$ of $\mathcal{U}$, we define
\begin{equation} \label{Eqn:CondEnt}
H(\lambda, \alpha \mid \cY) = \int H(\lambda_y, \alpha) \, d\nu(y).
\end{equation}

Now suppose $\nu \in \mathcal{M}(\cY,T)$.
It's possible to show (see, e.g., \cite{Kifer2001}) that if $\lambda \in \cJ(R: \nu)$ 
and $\lambda = \int \lambda_y \otimes \delta_y \, d\nu(y)$ is its disintegration over $\nu$, then for every finite measurable partition $\alpha$ of $\mathcal{U}$ the following limit exists:
\begin{equation*}
h^{\nu}(\lambda,\alpha) \, := \, 
\lim_n  \frac{1}{n} H(\lambda, \alpha_n \mid \cY) \, = \, \inf_n \frac{1}{n} H( \lambda, \alpha_n \mid \cY) \, d\nu(y),
\end{equation*}
where $\alpha_n = \bigvee_{k=0}^{n-1} R^{-k} \alpha$.
Furthermore, when $\lambda$ is ergodic, it can be shown (again see \cite{Kifer2001}) that 
for $\nu$ almost every $y$,
\begin{equation*}
h^{\nu}(\lambda,\alpha) = \lim_n \frac{1}{n} H(\lambda_y,\alpha_n).
\end{equation*}
The {\it fiber entropy} of $\lambda$ over $\nu$ is defined as $h^{\nu}(\lambda) = \sup_{\alpha} h^{\nu}(\lambda,\alpha)$, where the supremum is taken over all finite measurable partitions $\alpha$ of $\mathcal{U}$. 
Note that the supremum may also be taken over partitions with diameter less than any $\epsilon >0$.
The fiber entropy $h^{\nu}(\lambda)$ quantifies the relative entropy of $\lambda$ over $\nu$. 

\subsection{Divergence terms} \label{Sect:DivergenceTerms}

Consider a parameter $\theta \in \Theta$ and a joining $\lambda \in \mathcal{J}(S : \nu)$. We would like to quantify the divergence of the joining $\lambda$ to the product measure $\mu_{\theta} \otimes \nu$, as it will play a role in the rate function $V$.  (Note that the measure $\mu_{\theta} \otimes \nu$ may be interpreted as a prior distribution on $\mathcal{X} \times \mathcal{Y}$ given $\theta$, as the prior on $\mathcal{X}$ is assumed to be independent of the observations.) However, the standard KL-divergence is insufficient for our purposes, since any two ergodic measures for a given system are known to be mutually singular, and hence their KL-divergence will be infinite. Instead, we make the following definitions, which are more suitable for dynamical systems.

Given two Borel probability measures $\eta$ and $\gamma$ on a compact metric space $\mathcal{U}$ and a finite measurable partition $\alpha$ of $\mathcal{U}$, we write $\eta \prec_{\alpha} \gamma$ whenever $\gamma(C) =0$ implies that $\eta(C) = 0$ for $C \in \alpha$. Let
\begin{equation*}
KL( \eta : \gamma \mid \alpha) = 
\left\{ \begin{array}{ll} \sum_{C \in \alpha} \eta(C) \log\frac{\eta(C)}{\gamma(C)}, & \text{ if } \eta \prec_{\alpha} \gamma \\
                                                                                 +\infty, & \text{ otherwise},
\end{array} \right.
\end{equation*}
where $0 \cdot \log \frac{0}{x} = 0$ for any $x$ by convention.
Note that $KL(\eta : \gamma \mid \alpha)$ is the KL-divergence from $\gamma$ to $\eta$ with respect to the partition $\alpha$, which is nonnegative.


Now consider a H\"{o}lder continuous potential $f : \mathcal{X} \to \R$ on a mixing SFT $\mathcal{X}$ with associated Gibbs measure $\mu \in \mathcal{M}(\mathcal{X},S)$. Let $\alpha$ be the partition of $\mathcal{X}$ into cylinder sets of the form $x[0]$ for some 
$x \in \mathcal{X}$, and let $\eta \in \mathcal{M}(\mathcal{X},S)$ be ergodic. 
In this situation, it is known \cite{Chazottes1998} that 
\begin{equation*}
\lim_n \frac{1}{n}KL( \eta : \mu \mid \alpha_n) = \mathcal{P}(f) - \biggl( h(\eta) + \int f \, d\eta \biggr),
\end{equation*}
where we recall that $\mathcal{P}(f)$ is the pressure of $f$, the partition $\alpha_n$ is defined to be $\bigvee_{k=0}^{n-1} S^{-k} \alpha$, and $h(\eta)$ is the entropy of $\eta$ with respect to $S$.
Next we generalize this result to handle the relative situation, which involves joinings and relative entropy. 

\begin{lemma} \label{Lemma:Dinterpretation}
Let $f : \mathcal{X} \to \R$ be a H\"{o}lder continuous potential on a mixing SFT $\mathcal{X}$ with associated Gibbs measure $\mu$. Let $\alpha$ be the partition of $\mathcal{X}$ into cylinder sets of the form $x[0]$, and let $\lambda \in \mathcal{J}(S : \nu)$ be ergodic. Then for $\nu$-almost every $y \in \mathcal{Y}$,
\begin{equation*}
\lim_n \frac{1}{n} KL( \lambda_y : \mu \mid \alpha_n) = \mathcal{P}(f) - \biggl( h^{\nu}(\lambda) + \int f \, d\lambda \biggr).
\end{equation*}
\end{lemma}
We defer the proof of Lemma \ref{Lemma:Dinterpretation} to Section \ref{Sect:DivergenceTerm}.
Based on this lemma, we make the following definition.
\begin{definition} \label{Def:D}
 Let $\mathcal{X}$ be a mixing SFT, $f : \mathcal{X} \to \R$ a H\"{o}lder continuous function, and $\mu$ the associated Gibbs measure. Further, let $(\mathcal{Y},T,\nu)$ be an ergodic system. Then define the relative divergence rate of $\lambda \in \mathcal{J}(S : \nu)$ to $\mu$ to be
\begin{equation} \label{Eqn:Harry}
D( \lambda : \mu) = \mathcal{P}(f) - \biggl( h^{\nu}(\lambda) + \int f(x) \, d\lambda(x,y) \biggr).
\end{equation}
\end{definition}
In the present setting, $D(\lambda : \mu)$ is always finite, and one may check that it is also nonnegative (see Lemma \ref{Lemma:JoiningsVP}).

\subsection{The rate function} \label{Sect:RateFunction}

In this section we define and discuss the rate function $V : \Theta \to \R$ whose existence is guaranteed by Theorem \ref{Thm:Pressure}. 
\begin{definition} \label{Defn:V}
For $\theta \in \Theta$, let
\begin{equation*}
V(\theta) = \inf \biggl\{ \int \ell \, d\lambda + D(\lambda : \mu_{\theta}) : \lambda \in \mathcal{J}( S : \nu) \biggr\}.
\end{equation*}
\end{definition}
Note that the variational expression defining $V$ contains the sum of an expected loss term and a divergence term. It is known that Bayesian posterior distributions satisfy a similar variational principle in the finite sample setting (see \cite{Jiang2008,Zhang2006a,Zhang2006b}). Our results show that this interpretation passes to the limit as the number of samples tends to infinity.

By Proposition \ref{Prop:LSC}, which appears in Section \ref{Sect:Regularity}, we have that $V$ is lower semi-continuous. 
Since the loss function is continuous, the proof the Proposition \ref{Prop:LSC} essentially follows from the the upper semi-continuity of the fiber entropy on the space of joinings $\mathcal{J}(S : \nu)$. 

\begin{remark}
Consider the introduction of an inverse temperature parameter $\beta \in \mathbb{R}$, as discussed in Remark \ref{Rmk:GroundStates}, and let $\ell_{\beta} = \beta \cdot \ell$ be the associated loss function. If we let $V_{\beta}$ be the associated rate function, then we see from Definition \ref{Defn:V} that
\begin{equation*}
V_{\beta}(\theta) =  \inf \biggl\{ \beta \cdot \int \ell \, d\lambda + D(\lambda : \mu_{\theta}) : \lambda \in \mathcal{J}( S : \nu) \biggr\}.
\end{equation*}
Dividing by $\beta$ and letting $\beta$ tend to infinity to investigate the ground state behavior, it is clear that the associated variational expression is
\begin{equation*}
V_{\infty}(\theta) := \lim_{\beta \to \infty} \frac{V_{\beta}(\theta)}{\beta} = \inf \biggl\{ \int \ell \, d\lambda  : \lambda \in \mathcal{J}( S : \nu) \biggr\}.
\end{equation*}
Interestingly, this variational expression has been studied recently as part of an asymptotic analysis of estimators based on empirical risk minimization for dynamical systems \cite{McGoff2016variational,McGoff2016empirical}. Indeed, the solution set $\Theta_{\infty}$ of this ground state variational problem exactly characterizes the set of possible limits of parameter estimates that asymptotically minimize average empirical risk.
\end{remark}

\section{Connections to previous work}

In the setting of i.i.d. samples, Doob \cite{Doob49CNRS} established Bayesian posterior consistency for almost every parameter 
value in the support of the prior using Martingale methods. 
Later, Schwartz \cite{Schwartz1965} gave necessary and sufficient conditions for posterior consistency at individual parameter values in the i.i.d.~setting; these conditions require that the prior charge all KL-neighborhoods of the parameter and that there exist a sequence of tests giving exponential separation of the parameter from other parameters. The challenges and pitfalls of proving posterior consistency for nonparametric models were highlighted by Diaconis and Freedman in \cite{diaconis1998}. The negative results motivated much of the recent work in Bayesian nonparametrics, as well as studying convergence in other metrics on the space of probability distributions (such as Hellinger), and consideration of rates of convergence. For a detailed review of the Bayesian nonparametic literature we refer the reader to the recent book by Ghosal and van der Vaart \cite{Ghosal2017}.

Recent years have witnessed substantial interest in moving beyond the i.i.d.~setting and considering statistical inference for dependent processes, including processes arising from dynamical systems. Statistical problems receiving recent attention in the context of dynamical systems include denoising (or filtering) \cite{Lalley1999,LalleyNobel2006}, consistency of maximum likelihood estimation \cite{McGoff2015}, forecasting and density estimation \cite{Hang2017,Steinwart2009}, empirical risk minimization \cite{McGoff2016variational,McGoff2016empirical}, and data assimilation and uncertainty quantification \cite{Law2015}. For a survey of this area, see \cite{McGoffSurvey2015}.
Bayesian posterior consistency for dependent processes has also received attention in the literature. In particular, posterior consistency has been established for certain families of finite state hidden Markov chains \cite{Chopinetal2015,Doucetal2016,Gassiat2014,Vernet2015}.  

The idea of a variational formulation of Bayesian inference was developed by Zellner \cite{Zellner1988} and the link between statistical mechanics and information theory with Bayesian inference was at the heart of the inference framework advocated by Edwin T. Jaynes \cite{Jaynes1973-JAYTWP}, a perspective that influenced Zellner \cite{Zellner1988}. Formulating Bayesian inference as a variational problem for infinite dimensional problems has been explored in the control theory and inverse problems literatures \cite{Luetal,Mitter2,1742-5468-2012-11-P11008}.  In  \cite{1742-5468-2012-11-P11008}  a variational formulation of Bayesian inference was developed for the problem of channel coding using ideas from statistical mechanics. In \cite{Mitter2} a variational characterization of Bayesian nonlinear estimation was shown to take the same form Gibbs measures in statistical mechanics. In \cite{Luetal}, the authors studied the inference problem of finding the most likely path given a Brownian dynamics model from molecular dynamics, which takes the form of a gradient flow in a potential, subject to small thermal fluctuations. In this problem setup, a variational solution was proposed for Bayesian inference.

The setting and results of \cite{McGoff2016variational} and \cite{McGoff2016empirical} are worthy of some discussion, as they may be considered frequentist analogues of the present work. Indeed, the setting of this previous work involves observations from an unknown ergodic system, a model family consisting of topological dynamical systems, and a loss function connecting the models to the observations, as in the present work. Given this setting, the previous work analyzes the convergence of parameter estimates obtained by empirical risk minimization, whereas we study the convergence of parameter estimates based on Bayesian updates (in the form of the Gibbs posterior). One additional difference is that the previous results on empirical risk minimization are more general, in the sense that the model families need only consist of continuous maps on compact metric spaces; whereas, in our Bayesian setting, we specialize to the case of SFTs with Gibbs measures. This focus on Gibbs measures in the Bayesian setting arises precisely because Gibbs measures satisfy the necessary exponential estimates (the Gibbs property (\ref{Eqn:GibbsProp})) to make the asymptotic analysis work. It should be noted that a Bayesian framework provides estimates of uncertainty which empirical risk minimization does not. 

The Gibbs posterior principle can be derived from general principles as a valid method to update belief distributions in the presence of a loss function \cite{Bissiri2016}. In particular, this framework for updating beliefs remains valid when one does not have access to a true likelihood. This inference framework has also been shown to have advantages in some settings \cite{Jiang2008}. One of the motivations for the use of the Gibbs posterior in \cite{Jiang2008} was that exponentiating a robust loss function can better accommodate model misspecification, e.g., when the assumed likelihood is not the sample generating process. A logistic regression example is provided in  \cite{Jiang2008} for which the usual posterior-based logistic regression produces suboptimal classification error even from among the misspecified logistic regression models, while the Gibbs posterior is optimal. Another argument for using a loss-based approach comes from the robust statistics literature \cite{Huber1981}. A key idea in robust statistics is that one can define loss functions that are not sensitive to contamination of standard error or likelihood models. Thus, even if the model is misspecified, inference using the robust loss function is still reliable. The advantage of the Gibbs posterior framework is that one can specify coherent Bayesian updating using a robust loss function and not have to specify the data generating process.

The thermodynamic formalism in dynamical systems, originally pioneered by Sinai, Ruelle, and Bowen, involves adaptation of many ideas and methods from statistical physics to the setting of dynamical systems, and it has played a large role in the development of ergodic theory and dynamical systems over many years. For an introduction to the area and some connections to statistical physics, see the books by Bowen \cite{Bowen1975}, Ruelle \cite{Ruelle2004}, or Walters \cite{Walters1982}. Let us mention that connections to Markov chains and other stochastic processes have a long history in this area \cite{Benedicks2000,Young1998,Young1999}. Additionally, relative equilibrium states were studied by Ledrappier and Walters \cite{LedrappierWalters1977}, and recent results on uniqueness of relative equilibrium states \cite{Allahbakhshi2013,Allahbakhshi2017,Antonioli2016,Petersen2003} may contain interesting ideas to apply towards Bayesian posterior consistency.

\section{Technical preliminaries}

This section contains several technical results that we use later in the proofs of the main theorems.

\subsection{Pressure Lemma}

We refer to the following elementary fact, which is an easy consequence of Jensen's inequality, as the Pressure Lemma; see \cite[Lemma 9.9]{Walters1982}.
\begin{lemma}
\label{Lem:Pressure}
Let $a_1, \dots, a_k$ be real numbers. If $p_i \geq 0$ and $\sum_{i=1}^k p_i =1$, then
\begin{equation*}
\sum_{i=1}^k p_i (a_i - \log p_i) \leq \log \Biggl( \sum_{i=1}^k \exp(a_i) \Biggr),
\end{equation*}
with equality if and only if
\begin{equation*}
p_i = \frac{ \exp(a_i) }{ \sum_{j=1}^k \exp(a_i)}.
\end{equation*}
\end{lemma}

\subsection{The space of joinings and the ergodic decomposition} \label{Sect:ErgDecomp}


Our proofs rely on a general version of the ergodic decomposition for invariant probability measures. 
The following version, a restatement of \cite[Theorem 2.5]{Sarig2008}, is sufficient for our purposes.


\begin{theorem*}[The Ergodic Decomposition]
Suppose that $R : \cU \to \cU$ is a Borel measurable map of a Polish space $\cU$ and that $\mu \in \cM(\cU,R)$. 
Then there exists a Borel probability measure $Q$ on $\cM(\cU)$ 
such that 
\renewcommand{\theenumi}{\arabic{enumi}}
\renewcommand{\labelenumi}{(\theenumi)}
\begin{enumerate}
\vskip.07in
\item $Q\bigl( \{ \eta \mbox{ is invariant and ergodic for $R$}  \} \bigr) = 1$
\vskip.1in
\item If $f \in L^1(\lambda)$, then $f \in L^1(\eta)$ for $Q$-almost every $\eta$, and
        \begin{equation*}
         \int f \, d\mu = \int \biggl(\int f \, d\eta \biggr) \, dQ(\eta).
        \end{equation*}
\end{enumerate}
Whenever (2) holds, we write $\mu = \int \eta \, dQ$.
\end{theorem*}


Additionally, we require the following results about the structure of $\mathcal{J}(S : \nu)$ from \cite{McGoff2016variational}.
\begin{theorem*}[Structure of the space of joinings]
 Suppose $R : \mathcal{U} \to \mathcal{U}$ is a continuous map of a compact metrizable space and $(\mathcal{Y},T,\nu)$ is an ergodic measure-preserving system as in Section \ref{Sect:Intro}. Then $\mathcal{J}(R : \nu)$ is non-empty, compact, and convex. Furthermore, a joining $\lambda \in \mathcal{J}(R : \nu)$ is an extreme point of $\mathcal{J}(R : \nu)$ if and only if $\lambda$ is ergodic for $R \times T$. Lastly, if $\lambda \in \mathcal{J}(R : \nu)$ and $\lambda = \int \eta \, dQ$ is its ergodic decomposition, then $Q$-almost every $\eta$ is in $\mathcal{J}(R : \nu)$.
\end{theorem*}

Let $\lambda \in \mathcal{J}(R : \nu)$. By the above theorem, the ergodic decomposition of $\lambda$ is a representation of $\lambda$ as an integral combination of the extreme points of $\mathcal{J}(R : \nu)$. 
A function $F : \mathcal{J}(R : \nu) \to \mathbb{R}$ is called \textit{harmonic} if for each $\lambda \in \mathcal{J}(R : \nu)$, 
\begin{equation*}
F(\lambda) = \int F(\eta) \, dQ(\eta),
\end{equation*}
where $\lambda = \int \eta \, dQ$ is the ergodic decomposition of $\lambda$.

\subsection{Disintegration results} \label{Sect:DisintegrationResults}

Suppose $R : \mathcal{U} \to \mathcal{U}$ is a continuous map of a compact metric space and $(\mathcal{Y},T,\nu)$ is an ergodic system.
It is well known in ergodic theory (see \cite{Glasner2003}) that for any joining $\lambda \in \mathcal{J}(R : \nu)$, if $\lambda = \int \lambda_y \otimes \delta_y \, d\nu(y)$ is its disintegration over $\nu$, then the family of measures $\{ \lambda_y\}_{y \in \mathcal{Y}}$ satisfies an additional invariance property, which we state in the following lemma.
\begin{lemma} \label{Lemma:InvarianceOfDisintegration}
Let $\lambda \in \cJ(R: \nu)$, and let $\lambda = \int \lambda_y \otimes \delta_y \, d\nu(y)$ be its disintegration over $\nu$.
Then $(\lambda_y \otimes \delta_y) \circ (R \times T)^{-1} = \lambda_{Ty} \otimes \delta_{Ty}$ for $\nu$-almost every $y \in \mathcal{Y}$, and hence, for every $f \in L^1(\lambda)$ and $\nu$-almost every $y \in \mathcal{Y}$,
\begin{equation*}
\int f(Ru,Ty) \, d\lambda_y(u) = \int f(u,Ty) \, d\lambda_{Ty}(u).
\end{equation*}
\end{lemma}

\subsection{Limiting average loss}

The following lemma will be applied to the limiting average loss. Recall that when $R : \mathcal{U} \to \mathcal{U}$ is a continuous map of a compact metric space, the space $\cJ(R : \nu)$ of joinings is non-empty. For notation, if $f : \mathcal{U} \times \mathcal{Y} \to \R$, then we let $f_n(u,y) = \sum_{k=0}^{n-1} f(R^ku, T^ky)$.

\begin{lemma} \label{Lemma:IntegralDisintegration}
Suppose that $R : \mathcal{U} \to \mathcal{U}$ is a Borel self-map of a complete metric space $\mathcal{U}$, 
and that $f : \mathcal{U} \times \mathcal{Y} \to \R$ is a Borel function for which there exists 
$f^* : \mathcal{Y} \to \R$ in $L^1(\nu)$ such that $\sup_{u \in U} |f(u,y)| \leq f^*(y)$ for each $y \in \mathcal{Y}$. 
Then for any joining $\lambda \in \cJ(R: \nu)$ with disintegration $\lambda = \int \lambda_y \otimes \delta_y \, d\nu(y)$ over $\nu$, 
 for $\nu$-almost every $y \in \cY$,
\begin{equation*}
\lim_n \frac{1}{n} \int f_n(u,y) \, d\lambda_y(u) = \int f \, d\lambda.
\end{equation*}
\end{lemma}
\begin{proof}
For $y \in \cY$ define $\tilde{f}(y) = \int f(u,y) \, d\lambda_y(u)$. Then $\tilde{f} \in L^1(\nu)$, since $f \in L^1(\lambda)$ (using the hypotheses involving $f^*$). Now Lemma \ref{Lemma:InvarianceOfDisintegration}, together with the pointwise ergodic theorem, yields that for $\nu$ almost every $y$,
\begin{equation*}
\lim_n \frac{1}{n} \int f_n(u,y) \, d\lambda_y(u) = \lim_n \frac{1}{n} \sum_{k=0}^{n-1} \tilde{f}\bigl( T^ky \bigr) = \int \tilde{f} \, d\nu = \int f \, d\lambda.
\end{equation*}
\end{proof}

\subsection{Fiber entropy} \label{Sect:FiberEntropy}

We require two additional properties of the fiber entropy in our setting.
The first property is that fiber entropy is harmonic. This fact appears with proof as Lemma 3.2 (iii) in \cite{LedrappierWalters1977} in a setting under which $T : \mathcal{Y} \to \mathcal{Y}$ is a continuous map of a compact space, but careful inspection shows that the proof does not depend on this hypothesis.

\begin{lemma} \label{Lemma:FiberEntropyIsHarmonic}
The map $\lambda \mapsto h^{\nu}(\lambda)$ from $\mathcal{J}(R : \nu)$ to the non-negative extended reals satisfies the following property: if $\lambda = \int \eta \, dQ(\eta)$ is the ergodic decomposition of $\lambda$, then
\begin{equation*}
h^{\nu}(\lambda) = \int h^{\nu}(\eta) \, dQ(\eta).
\end{equation*}
\end{lemma}

Next, we note that fiber entropy function is upper semi-continuous in our setting. The proof of Lemma 2.2 in  \cite{Walters1986} establishes upper semi-continuity of fiber entropy in a setting closely related to ours. By making only minor modifications of that proof, one may adapt it to our setting and prove the following lemma.

\begin{lemma} \label{Lemma:FiberEntropyIsUSC}
Let $\Theta$, $(\mathcal{X},S)$, and $(\mathcal{Y},T,\nu)$ be as in the introduction, and let $R = \Id \times S$ act on the product space
$\mathcal{U} = \Theta \times \mathcal{X}$. Then the map $\lambda \mapsto h^{\nu}(\lambda)$ 
from $\mathcal{J}(R : \nu)$ to $\R$ is upper semi-continuous. 
\end{lemma}

\subsection{Divergence terms and average information} \label{Sect:DivergenceTerm}

Define
\begin{equation*}
L( \eta : \gamma \mid \alpha) = \left\{ \begin{array}{ll} 
                                                        \sum_{C \in \alpha} \eta(C) \log \gamma(C), & \text{ if } \eta \prec_{\alpha} \gamma \\
                                                        -\infty, & \text{ otherwise},
                                                \end{array}
                                                \right.
\end{equation*}
where $0 \cdot \log 0 = 0$ by convention. With these definitions, we always have
\begin{equation} \label{Eqn:Rain}
- KL( \eta : \gamma \mid \alpha) = H(\eta, \alpha) + L(\eta : \gamma \mid \alpha).
\end{equation}
Recall that $H(\eta,\alpha)$ may be interpreted as the expected information of $\eta$ under the partition $\alpha$, 
where the expectation is with respect to $\eta$. 
In contrast, $-L(\eta : \gamma \mid \alpha)$ may be interpreted as the expected information of $\gamma$ under 
the partition $\alpha$, where the expectation is again taken with respect to $\eta$. 
In what follows, if $\alpha$ is a partition of a space $\mathcal{U}$ and $u \in \mathcal{U}$, we 
let $\alpha(u)$ denote the partition element containing $u$.
Here we restate and then prove Lemma \ref{Lemma:Dinterpretation}.

\begin{lemma} \label{Lemma:DinterpretationTakeTwo}
Let $f : \mathcal{X} \to \R$ be a H\"{o}lder continuous potential on a mixing SFT $\mathcal{X}$ with associated Gibbs measure $\mu$. Let $\alpha$ be the partition of $\mathcal{X}$ into cylinder sets of the form $x[0]$, and let $\lambda \in \mathcal{J}(S : \nu)$ be ergodic. Then for $\nu$-almost every $y \in \mathcal{Y}$,
\begin{equation*}
\lim_n \frac{1}{n}KL( \lambda_y : \mu \mid \alpha_n) = \mathcal{P}(f) - \biggl( h^{\nu}(\lambda) + \int f \, d\lambda \biggr).
\end{equation*}
\end{lemma}

\begin{proof}
Recall that by the Gibbs property for $\mu$, for any $n \geq 1$ and $x$ in $\mathcal{X}$, we have
\begin{equation*}
K^{-1} \leq \frac{\mu(\alpha_n(x))}{- n \, \mathcal{P}(f) + \sum_{k=0}^{n-1} f \circ S^k(x)} \leq K.
\end{equation*}
Taking logarithms yields the bound
\begin{equation*}
\biggl| \log\bigl(\mu(\alpha_n(x))\bigr) - \bigl( - n \, \mathcal{P}(f) + f_n(x) \bigr) \biggr| \leq \log K.
\end{equation*}
As this inequality is uniform in $x$, we may integrate with respect to $\lambda_y$ to obtain
\begin{equation*}
\biggl| L(\lambda_y : \mu \mid \alpha_n) - \biggl( - n \, \mathcal{P}(f) + \int f_n \, d\lambda_y \biggr) \biggr| \leq \log K.
\end{equation*}
Dividing by $n$ and applying Lemma \ref{Lemma:IntegralDisintegration} gives
\begin{equation} \label{Eqn:Car}
\lim_n \frac{1}{n} L(\lambda_y : \mu \mid \alpha_n) = - \mathcal{P}(f) + \int f \, d\lambda.
\end{equation}

It follows from (\ref{Eqn:Rain}) that
$
KL( \lambda_y : \mu \mid \alpha_n) = - H( \lambda_y, \alpha_n) - L( \lambda_y : \mu \mid \alpha_n)
$.
Since $\lambda$ is ergodic, for $\nu$-almost every $y$, we have 
$n^{-1} H(\lambda_y, \alpha_n) \to h^{\nu}(\lambda,\alpha) = h^{\nu}(\lambda)$,
where the equality is a result of the fact that $\alpha$ is a generating partition for $(\mathcal{X},S)$.
Combining this fact with (\ref{Eqn:Car}), we find that for $\nu$-almost every $y$,
\begin{align*}
\lim_n \frac{1}{n} KL( \lambda_y : \mu \mid \alpha_n)  & = - h^{\nu}( \lambda) - \biggl( - \mathcal{P}(f) + \int f \, d\lambda \biggr)  
\end{align*}
as desired.
\end{proof}

Now we prove a lemma that guarantees that $D(\lambda : \mu_{\theta}) \geq 0$.

\begin{lemma} \label{Lemma:JoiningsVP}
For each $\theta \in \Theta$ and $\lambda \in \mathcal{J}(S,\nu)$,
\begin{equation*}
\int f_{\theta} \, d\lambda + h^{\nu}(\lambda) \leq \mathcal{P}(f_{\theta}).
\end{equation*}
\end{lemma}

\begin{proof}
Let $\mu$ be the $\mathcal{X}$-marginal of $\lambda$. 
Then $h^{\nu}(\lambda) \leq h^{\nu}(\mu \otimes \nu) = h(\mu)$, where the inequality follows
from elementary information theoretic facts concerning conditional entropy  (see \cite{Cover2012})
and the equality is a basic property of fiber entropy.
Then by the variational principle for pressure (\ref{Eqn:VPTP}),
\begin{equation*}
\int f_{\theta} \, d\lambda + h^{\nu}(\lambda) \leq \int f_{\theta} \, d\mu + h(\mu) \leq \mathcal{P}(f_{\theta}),
\end{equation*}
as desired.
\end{proof}



We now establish a lemma that is used in the proof of Theorem \ref{Thm:Pressure}. 
This result allows us to approximate the expected information in the prior $P_0$, where the expectation is with respect to an 
arbitrary measure, in terms of an average of a continuous function. 
These types of estimates are available precisely because our model class consists of 
Gibbs measures: indeed, they do not hold for arbitrary invariant measures for dynamical systems.

For any Borel probability measure $\eta$ on $\Theta \times \cX$, let $\eta_n$ denote its time-average up to time $n$:
\begin{equation*}
\eta_n(E) = \frac{1}{n} \sum_{k=0}^{n-1} \eta\bigl( (\Id \times S)^{-k}E \bigr),
\end{equation*}
where $\Id : \Theta \to \Theta$ is the identity.

\begin{lemma} \label{Lemma:Sunshine}
Let $K$ be the constant in the uniform Gibbs property (\ref{Eqn:UniformGibbs}). 
For any $\epsilon >0$ there exists $\delta >0$ such that if the diameter of $\alpha$ is less than $\delta$ and $\beta$ is the partition of $\mathcal{X}$ into cylinder sets of the form $x[0]$, then for any Borel probability measure $\eta$ on $\Theta \times \cX$, and any $n \geq 0$, 
\begin{equation*}
\biggl| \frac{1}{n} L( \eta : P_0 \mid \alpha \times \beta_n ) - \frac{1}{n} \int \bigl(f_{\theta}(x) - \mathcal{P}(f_{\theta}) \bigr) \, d\eta_n(\theta,x)\biggr|  \leq \epsilon + \frac{ \log K}{n} .
\end{equation*}
\end{lemma}

\begin{proof}
Let $\epsilon > 0$. By the uniform continuity of $f_{\theta}$ and $\mathcal{P}(f_{\theta})$ in $\theta$
and the uniform Gibbs property, there exists $\delta > 0$ such that if the diameter of $\alpha$ is less than 
$\delta$ and $\beta$ is the partition of $\mathcal{X}$ into sets of the form $x[0]$, then for all 
$\theta \in \Theta$, $x \in \mathcal{X}$, and $n \geq 1$,
\begin{equation*}
K^{-1} \exp(- \epsilon n) \leq 
\frac{P_0( (\alpha \times \beta_n) (\theta,x))}{\exp \biggl( - n \,\mathcal{P}(f_{\theta}) + \sum_{k=0}^{n-1} f_{\theta}(S^kx) \biggr) } 
\leq K \exp( \epsilon n).
\end{equation*}
Taking logarithms and dividing by $n$, we obtain the inequality
\begin{equation*}
\biggl|\frac{1}{n} \log P_0( (\alpha \times \beta_n) (\theta,x)) - 
\frac{1}{n} \sum_{k=0}^{n-1} f_{\theta}(S^kx) + \mathcal{P}(f_{\theta}) \biggr| \leq \epsilon + \frac{\log K}{n},
\end{equation*}
which is uniform over $(\theta,x) \in \Theta \times \cX$.
Now let $\eta$ be any Borel probability measure on $\Theta \times \cX$. Then by integrating with respect to $\eta$, we see that
\begin{align*}
\biggl|  \frac{1}{n} L( \eta : P_0 \mid \alpha \times \beta_n )  - \frac{1}{n} \int \bigl(f_{\theta}(x) - \mathcal{P}(f_{\theta}) \bigr) \, d\eta_n(\theta,x)\biggr| 
 \leq \epsilon + \frac{ \log K}{n} .
\end{align*}
\end{proof}

%
%

\section{Semicontinuity of the rate function and $\Theta_{\min}$} \label{Sect:Regularity}

\begin{proposition} \label{Prop:LSC}
The map $V : \Theta \to \R$ defined in Definition \ref{Defn:V} is lower semi-continuous, and hence the set $\Theta_{\min}$ is compact and non-empty. 
\end{proposition}

\begin{proof}
Let $\mathcal{U} = \Theta \times \mathcal{X}$ and let $R : \cU \to \cU$ be given by $R = \Id \times S$, where $\Id$ is
the identity on $\Theta$.
Define $\psi : \mathcal{U} \times \mathcal{Y} \to \R$ by 
\begin{equation*}
\psi(\theta,x,y) = -\ell(\theta,x,y)  + f_{\theta}(x) - \mathcal{P}(f_{\theta}),
\end{equation*}
which is continuous and satisfies $\sup_{u \in \cU} |\psi(u,y)| \leq \psi^* \in L^1(\nu)$.
Finally, define $F: \mathcal{J}(R : \nu) \to \R$ by
\begin{equation*}
F(\lambda) = \int \psi \, d\lambda + h^{\nu}(\lambda).
\end{equation*}
Since $\psi$ is continuous and $h^{\nu}$ is upper semi-continuous (by Lemma \ref{Lemma:FiberEntropyIsUSC}), $F$ is upper semi-continuous. 
Let $\proj_{\Theta} : \mathcal{J}(R : \nu) \to \mathcal{M}(\Theta)$ be defined by setting 
$\proj_{\Theta}(\lambda)$ to be the $\Theta$-marginal of $\lambda$, which is a continuous surjection of 
compact spaces. One may easily check from the definition of upper semicontinuity that the function
\begin{equation*}
\theta \mapsto \sup \bigl\{ F (\lambda) : \proj_{\Theta}(\lambda) = \delta_{\theta} \bigr\}
\end{equation*}
is also upper semicontinuous. 
Since $V(\theta)$ is the negative of this function, 
we conclude that $V$ is lower semi-continuous.

For the second part of the proposition, we note that $\Theta_{\min}$ is the $\argmin$ of the lower semi-continuous function $V$ on the compact set $\Theta$, and hence it is non-empty and compact.
\end{proof}

\vskip.1in

%
%
%

\section{Convergence of the partition function and a variational principle}

In this section, we prove Theorem \ref{Thm:Pressure}, which concerns the convergence of the average log
normalizing constant (partition function) $n^{-1} \log Z_n$. 
The starting point of the proof, which is an application of the Pressure Lemma, allows us to express the 
main statistical object, the Gibbs posterior distribution, as the solution of a variational problem involving information 
theoretic notions such as entropy and average information, which have long been studied in dynamics. 
The proof of Theorem \ref{Thm:Pressure} follows.

 To ease notation slightly in this section, we let $g = -\ell$ and $g_n = - \ell_n$, where $\ell_n$ is defined in (\ref{Eqn:Ellen}).
We also set $\cU = \Theta \times \mathcal{X}$ and $R(\theta,x) = (\theta,S(x))$.
For $\lambda \in \mathcal{J}( R : \nu)$, we will have use for the notation
\begin{equation*}
G(\lambda) =  \int \bigl( \mathcal{P}(f_{\theta}) - f_{\theta}(x) \bigr) \, d\lambda(\theta,x,y) -  h^{\nu}( \lambda ).
\end{equation*}
Although we do not use this fact, we note that $G(\lambda)$ can be written as an integral over $\theta$ of terms of the form $D( \lambda_{\theta}, \mu_{\theta})$ (as in Definition \ref{Def:D}).
Lemma \ref{Lemma:FiberEntropyIsHarmonic} ensures that $h^{\nu}(\cdot)$ is harmonic, and therefore
the same is true of $G : \mathcal{J}( R : \nu) \to \mathbb{R}$.
In this notation, our goal is to prove
\begin{equation*}
\lim_n \frac{1}{n} \log Z_n = \sup \biggl\{ \int g \, d\lambda - G(\lambda) : \lambda \in \mathcal{J}(R : \nu) \biggr\}.
\end{equation*}

We present the proof in two stages: first we establish that the expression in right-hand side is a lower bound for $\lim_n n^{-1} \log Z_n$, and then we prove that the same expression provides an upper bound.

\subsection{Lower bound} \label{Sect:LB}

The goal of this section is to prove the following result.
\begin{proposition} \label{Prop:LB}
 For $\nu$-almost every $y \in \mathcal{Y}$,
 \begin{equation*}
  \lim_n \frac{1}{n} \log Z_n \geq \sup \biggl\{ \int g \, d\lambda - G(\lambda) : \lambda \in \mathcal{J}(R : \nu) \biggr\},
 \end{equation*}
 where $Z_n = Z_n(y)$.
\end{proposition}

Before proving this proposition, we first establish a lemma.
If $\eta$ is a Borel probability measure on $\Theta \times \cX$ and $\eta(C)>0$, 
then let $\eta_C$ denote the conditional distribution $\eta( \cdot \mid C)$.
Also, we say that $\beta$ is a partition of $\mathcal{X}$ according to central words whenever $\beta = \{ [x_{-m}^m] : x \in \mathcal{X} \}$ for some $m \geq 0$.

\begin{lemma} \label{Lemma:LBpartition}
Let $\alpha$ be a finite measurable partition of $\Theta$ with $\diam(\alpha)<\delta$, 
and let $\beta$ be a partition of $\mathcal{X}$ according to central words 
such that $\diam(\beta) < \delta$. Then for any Borel probability measure 
$\eta$ on $\Theta \times \cX$, any $y \in \mathcal{Y}$, and any $n \geq 1$,
\begin{align*}
 \int g_n(\theta,x,y)  & \, d\eta(\theta,x) - KL(\eta : P_0 \mid \alpha \times \beta_{n})  \\
 & \leq \log \biggl[ \int \exp( g_n(\theta,x,y) ) \, dP_0(\theta,x) \biggr] + \sum_{k=0}^{n-1} \rho_{\delta}\bigl(T^k y \bigr).
\end{align*}
where $\rho_{\delta}$ is the local difference function appearing in property (iii) of the loss.
\end{lemma}

\begin{proof}
If $\eta \nprec_{\alpha \times \beta_{n}} P_0$, then the inequality holds trivially. 
Now suppose $\eta \prec_{\alpha \times \beta_{n}} P_0$, and let $\xi = \{ C \in \alpha \times \beta_{n} : \eta(C) >0 \}$.
For $C \in \xi$ and $(\theta,x),(\theta',x') \in C$, property (iii) of the loss function, and our hypotheses on $\alpha$ and $\beta$ yield that
\begin{equation*}
g_n(\theta',x',y) \leq g_n(\theta,x,y) + \sum_{k=0}^{n-1} \rho_{\delta}\bigl(T^k y \bigr).
\end{equation*}
Integrating out $(\theta',x')$ with respect to the conditional distribution $\eta_C$ gives
\begin{equation*}
\int_C g_n(\theta',x',y) \, d\eta_C(\theta',x') \leq  g_n(\theta,x,y) +  \sum_{k=0}^{n-1} \rho_{\delta}\bigl(T^k y \bigr).
\end{equation*}
After exponentiation and integration with respect to the $P_{0,C}$, we get
\begin{align*} 
\exp \biggl(  \int_C g_n(\theta',x',y) \, & d\eta_C(\theta',x') \biggr) \\
& \leq \exp\Biggl( \sum_{k=0}^{n-1} \rho_{\delta}\bigl(T^k y \bigr)\Biggr) \int_C \exp \bigl( g_n(\theta,x,y) \bigr) \, dP_{0,C}(\theta,x).
\end{align*}
Invoking Lemma \ref{Lem:Pressure} and the inequality above, we find that 
\begin{align*}
 \int g_n(\theta',x',y) \, & d\eta(\theta',x') - KL( \eta : P_0 \mid \alpha \times \beta_{n})  \\
 & =  \sum_{C \in \xi} \eta(C) \biggl[  \log P_0(C) + \int_C g_n(\theta',x',y) \, d\eta_C(\theta',x') - \log \eta(C) \biggr] \\
& \leq \log \sum_{C \in \xi} \exp\biggl( \log P_0(C) +    \int_C g_n(\theta',x',y) \, d\eta_C(\theta',x') \biggr) \\
 & = \log \sum_{C \in \xi} \exp\biggl( \int_C g_n(\theta',x',y) \, d\eta_C(\theta',x') \biggr) P_0(C) \\
 & \leq \log \sum_{C \in \xi} \biggl(\int_C \exp\bigl(  g_n(\theta,x,y)  \bigr) \, dP_{0,C}(\theta,x) \biggr) P_0(C) + \sum_{k=0}^{n-1} \rho_{\delta}\bigl(T^k y \bigr) \\
 & = \log \int \exp\bigl(  g_n(\theta,x,y)  \bigr) \, dP_0(\theta,x) + \sum_{k=0}^{n-1}  \rho_{\delta}\bigl(T^k y \bigr),
 \end{align*}
as was to be shown.
\end{proof}

\begin{PfofPropLB}

Fix an ergodic joining $\lambda \in \cJ(R : \nu)$ and $\epsilon > 0$.
Let $\delta > 0$ be sufficiently small that the bound of Lemma \ref{Lemma:Sunshine} holds
and that $\int \rho_{\delta} \, d\nu < \epsilon$ (using property (iii) of the loss).
Fix a finite measurable partition $\alpha$ of $\Theta$ such that $\diam(\alpha) < \delta$, and select $m$ large enough so that the 
partition $\beta$ of $\cX$ generated by central words of length $m$ satisfies $\diam(\beta) < \delta$.
Then for $\nu$-almost every $y$,
\begin{align*}
H(\lambda_y, \alpha \times \beta_{n}) &  + L( \lambda_y : P_0 \mid \alpha \times \beta_{n}) + \int g_n \, d\lambda_y \\
& =  \int g_n  \, d\lambda_y - KL(\lambda_y : P_0 \mid \alpha \times \beta_{n})  \\
&\leq \log \int \exp( g_n(\theta,x,y) ) \, dP_0(\theta,x) + \sum_{k=0}^{n-1} \rho_{\delta}\bigl(T^k y \bigr) ,
\end{align*}
where the inequality follows from Lemma \ref{Lemma:LBpartition}. 
Dividing each side of the inequality above by $n$, and then letting $n$ tend to infinity,  
Lemma \ref{Lemma:IntegralDisintegration}, Lemma \ref{Lemma:Sunshine}, and the ergodic theorem together imply that
for $\nu$-almost every $y \in \mathcal{Y}$,
\begin{align*}
h^{\nu}( \lambda, \alpha \times \beta ) + 
\int \bigl(f_{\theta}(x) - \mathcal{P}(f_{\theta}) \bigr) \, d\lambda + \int g \, d\lambda  \leq \liminf_{n} \frac{1}{n} Z_n(y) + 2 \epsilon.
\end{align*}
Taking the supremum over all partitions $\alpha$ of $\Theta$ with diameter less than $\delta$ and 
all partitions $\beta$ of $\cX$ generated by central words of length at least $m$, we obtain the inequality
\begin{align*}
\int g \, d\lambda - G(\lambda)  \leq \liminf_{n} \frac{1}{n} \log Z_n(y) + 2 \epsilon.
\end{align*}
Since $\epsilon>0$ was arbitrary,
\begin{equation*}
 \int g \, d\lambda - G(\lambda)  \leq \liminf_{n} \frac{1}{n} \log Z_n.
\end{equation*}
As this inequality holds for all ergodic $\lambda \in \mathcal{J}(R : \nu)$ and the left-hand side is harmonic in $\lambda$, we have
\begin{equation*}
\sup_{\lambda \in \mathcal{J}(R : \nu)} \biggl\{ \int g \, d\lambda - G(\lambda) \biggr\}  \leq \liminf_{n} \frac{1}{n} \log Z_n,
\end{equation*}
which completes the proof.
\end{PfofPropLB}

\vskip.1in

\subsection{Upper bound} \label{Sect:UB}

In Proposition \ref{Prop:UB} below we establish an almost sure upper bound on the limiting behavior of 
$n^{-1} \log Z_n(y)$.  Together with the lower bound in Proposition \ref{Prop:LB}, this completes the 
proof of Theorem \ref{Thm:Pressure}.

\begin{proposition} \label{Prop:UB}
For $\nu$-almost every $y \in \mathcal{Y}$,
\begin{equation*}
\limsup_n \frac{1}{n} \log Z_n(y) \leq \sup_{\lambda \in \mathcal{J}(R : \nu)} \biggl\{ \int g \, d\lambda - G(\lambda) \biggr\}.
\end{equation*}
\end{proposition}

We begin with a preliminary lemma.  Recall that $P_0$ is the prior distribution on $\Theta \times \cX$ generated
by the prior $\pi_0$ (defined in (\ref{Eqn:BigPrior})) and the family $\{ \mu_\theta : \theta \in \Theta \}$, while
$P_n( \cdot \mid y)$ is the Gibbs posterior distribution associated with $y, Ty, \ldots, T^{n-1}y$ (defined in (\ref{Eqn:GibbsDist})).
To simplify notation, in what follows $P_n( \cdot \mid y)$ is denoted by $P_n^y$.

\begin{lemma} \label{Lemma:Ghostbusters}
If $\alpha$ is a finite measurable partition of $\Theta \times \cX$ with diameter less than $\delta$ then 
for $y \in \cY$ and $n \geq 1$,
\begin{align*}
\log \int & \exp( g_n(\theta,x,y) ) \, dP_0 \\
& \leq H(P_n^y, \alpha_n) + L(P_n^y : P_0 \mid \alpha_n) 
   + \int g_n(\theta,x,y) \, dP_n^y +  \sum_{k=0}^{n-1} \rho_{\delta}\bigl(T^k y \bigr).
\end{align*}
\end{lemma}

\begin{proof}
Let $\alpha$ be a finite measurable partition of $\Theta \times \cX$ with $\diam(\alpha)<\delta$, and let $y \in \cY$.
By definition $P_n^y$ and $P_0$ are equivalent measures, 
and hence $P_n^y \prec_{\alpha_n} P_0$ and $P_0 \prec_{\alpha_n} P_n^y$.
Let $\xi = \{ C \in \alpha_n : P_0(C) >0\} = \{ C \in \alpha_n : P_n^y(C) >0\}$.

Fix $C \in \xi$ for the moment.  For points $(\theta,x),(\theta',x') \in C$ the hypothesis on $\alpha$ ensures that
\begin{equation*}
g_n(\theta,x,y)  \leq  g_n(\theta',x',y) +  \sum_{k=0}^{n-1} \rho_{\delta}\bigl(T^k y \bigr)
\end{equation*}
where $\rho_\delta ()$ is defined in condition (iii) of the loss.
Exponentiating both sides of the inequality and integrating $(\theta,x)$ with respect to the prior
$P_{0,C}$ conditioned on being in $C$ yields
\begin{equation*}
\int_C  \exp(g_n(\theta,x,y)) \, dP_{0,C} \leq  \exp( g_n(\theta',x',y) ) \exp\biggl( \sum_{k=0}^{n-1} \rho_{\delta}\bigl(T^k y \bigr)\biggr).
\end{equation*}
Taking logarithms and integrating $(\theta',x')$ with respect to the posterior $P_{n,C}^y$ conditioned on being in $C$ yields
\begin{equation} \label{Eqn:Spooky}
\log \int_C  \exp(g_n(\theta,x,y)) \, dP_{0,C} \leq \int_C g_n(\theta',x',y) \, dP_{n,C}^y + \sum_{k=0}^{n-1} \rho_{\delta}\bigl(T^k y \bigr).
\end{equation}
By the definition of $P_n$ and Lemma \ref{Lem:Pressure}
we have
\begin{align*}
\log & \int \exp( g_n(\theta,x,y) ) \, dP_0 \\
& = \log \sum_{C \in \xi} \exp\biggl[ \log \int_C \exp ( g_n(\theta,x,y) ) \, dP_0 \biggr] \\
& = \sum_{C \in \xi} P_n^y(C) \biggl[ - \log P_n^y(C) + \log P_0(C) + \log \int_C \exp ( g_n(\theta,x,y) ) \, dP_{0,C} \biggr] \\
& = H(P_n^y, \alpha_n) + L(P_n^y : P_0 \mid \alpha_n)  +  
\sum_{C \in \xi} P_n^y(C) \log \int_C \exp ( g_n(\theta,x,y) ) \, dP_{0,C}.
\end{align*}
Applying inequality (\ref{Eqn:Spooky}) to the terms of the final sum above, we see that
\begin{align*}
\log & \int \exp( g_n(\theta,x,y) ) \, dP_0(\theta,x) \\
& \leq H(P_n^y, \alpha_n) + L(P_n^y : P_0 \mid \alpha_n) + \int g_n(\theta,x,y) \, dP_n^y(\theta,x) +  \sum_{k=0}^{n-1} \rho_{\delta}\bigl(T^k y \bigr)
\end{align*}
as desired.
\end{proof}

\vskip.1in

\begin{PfofPropUB}

To begin the proof, define
\begin{equation*}
\eta_n^y = \frac{1}{n} \sum_{k=0}^{n-1} (P_n^y \circ (\Id \times S)^{-k}) \otimes \delta_{y_k}.
\end{equation*}
By \cite[Lemma 2.1]{Kifer2001}, 
for $\nu$-almost every $y$, the sequence $\{\eta_n^y\}_n$ is tight and all of its limit points are contained in $\cJ(R : \nu)$. For a given $y$ in this set of full measure, let $\lambda$ be such a limit point, with $\eta_{n_k}^y \to \lambda$. 

Let $\epsilon >0$, and choose $\delta >0$ such that $\int \rho_{\delta} \, d\nu < \epsilon$. Choose a finite measurable partition $\alpha$ of $\Theta \times \cX$ such that $\diam(\alpha) < \delta$ and $\proj_{\Theta \times \cX}(\lambda)( \partial \alpha) = 0$ (which exists since $\Theta \times \cX$ is compact \cite[Lemma 8.5]{Walters1982}). By adapting an argument from \cite[p. 190]{Walters1982} involving subadditivity of measure-theoretic entropy, 
we obtain that for each $q \geq 1$, for $n \geq q$,
\begin{align*}
\frac{1}{n} H(P_n^y, \alpha_n) & \leq \frac{1}{n} \sum_{k=0}^{n-1} \frac{1}{q} H( P_n^y \circ R^{-k}, \alpha_q) + o(1) \\
 & = \frac{1}{q}H(\eta_n^y,\alpha_q \mid \cY) + o(1),
\end{align*}
where $o(1)$ refers to a term that tends to $0$ as $n$ tends to infinity (for fixed $q$). 
Then by letting $n$ tend to infinity and applying \cite[Lemma 2.1]{Kifer2001} again, we see that
\begin{equation*}
\limsup_n \frac{1}{n} H(P_n^y, \alpha_n) \leq \frac{1}{q} H(\lambda, \alpha_q \mid \cY),
\end{equation*}
where the conditional entropy $H( \cdot \mid \cY)$ is defined in (\ref{Eqn:CondEnt}).
To proceed with the proof, we require the following lemma. 
Recall that at the beginning of this section, we set $g = - \ell$ and $g_n = - \ell_n$. 


\begin{lemma} \label{Lemma:Groundhog}
Let $\{Q_n\}_n$ be any sequence of measures on $\Theta \times \cX$.  For each $n \geq 1$ 
and $y \in \cY$ define
\begin{equation*}
\eta_n = \frac{1}{n} \sum_{k=0}^{n-1} (Q_n \circ R^{-k}) \otimes \delta_{y_k}.
\end{equation*}
If the subsequence $\{\eta_{n_k}\}_k$ converges to $\lambda$, then
\begin{equation*}
\lim_k \frac{1}{n_k} \int g_{n_k}(\theta,x,y) \, dQ_{n_k}(\theta,x) = \int g \, d\lambda.
\end{equation*}
\end{lemma}
\begin{proof}
By definition of $\eta_n$,
\begin{align*}
\frac{1}{n} \int g_n(\theta,x,y) \, dQ_n(x) & = \frac{1}{n} \int \sum_{k=0}^{n-1} g(\theta,S^k x, y_k) \, dQ_n(x) \\
& = \int g \, d\eta_n.
\end{align*}
Then the desired limit follows from the fact that $\{\eta_{n_k}\}_k$ converges to $\lambda$ and 
$g$ is continuous.
\end{proof}

Combining Lemma \ref{Lemma:Groundhog} with Lemmas \ref{Lemma:Sunshine} and 
\ref{Lemma:Ghostbusters}, we find that for $\nu$-almost every $y \in \mathcal{Y}$
\begin{align*}
\limsup_n \frac{1}{n} & \log \int \exp( g_n(\theta,x,y) ) \, dP_0(\theta,x) \\
&  \leq  \frac{1}{q} H(\lambda,\alpha_q \mid \cY) + \int \bigl(f_{\theta}(x) -  \mathcal{P}(f_{\theta}) \bigr) \, d\lambda(\theta,x,y)  + \int g \, d\lambda + 2\epsilon.
\end{align*}
Letting $q$ tend to infinity, we get
\begin{align*}
\limsup_n \frac{1}{n} & \log \int \exp( g_n(\theta,x,y) ) \, dP_0(\theta,x) \\
& \leq  h^{\nu}(\lambda,\alpha) + \int \bigl(f_{\theta}(x) -  \mathcal{P}(f_{\theta})\bigr) \, d\lambda(\theta,x,y)  + \int g \, d\lambda + 2\epsilon \\
& \leq h^{\nu}(\lambda) + \int \bigl(f_{\theta}(x) -  \mathcal{P}(f_{\theta}) \bigr) \, d\lambda(\theta,x,y)  + \int g \, d\lambda + 2\epsilon.
\end{align*}
Since $\epsilon$ was arbitrary, we obtain
\begin{align*}
\limsup_n \frac{1}{n} & \log \int \exp( g_n(\theta,x,y) ) \, dP_0(\theta,x) \\
 & \leq  \int g \, d\lambda - G(\lambda) \\
 & \leq \sup \biggl\{  \int g \, d\lambda - G(\lambda)  : \lambda \in \mathcal{J}(R : \nu) \biggr\}.
\end{align*}
This concludes the proof of Proposition \ref{Prop:UB}.
\end{PfofPropUB}



%
%

\section{Convergence of Gibbs posterior distributions}

The purpose of this section is to establish Theorem \ref{Thm:PosteriorConvergence} concerning convergence of the Gibbs posterior distributions to the solution set of a variational problem. From the dynamics point of view, this convergence highlights the role of the variational problem and the associated equilibirum joinings. We believe these objects to be worthy of further study. From the statistical point of view, this result describes the concentration of posterior distributions, which is of interest in any frequentist analysis of Bayesian methods. The proof follows somewhat directly from Theorem \ref{Thm:Pressure}. 

  
  \vspace{2mm}

\begin{PfofPosteriorConvergence}

Let $U$ be an open neighborhood of $\Theta_{\min}$. Let $F = \Theta \setminus U$, which is closed and therefore compact. If $\pi_0(F) = 0$, then $\pi_n(F \mid y) =0$ for all $n$. Now suppose $\pi_0(F) >0$, and let $\tilde{\pi}_0 = \pi_0( \cdot \mid F)$ be the conditional prior on $F$. Let $V_*$ be the common value of $V(\theta)$ for 
$\theta \in \Theta_{\min}$.  As $V : \Theta \to \R$ is lower semi-continuous and $F$ is compact and disjoint from $\Theta_{\min}$, there exists $\epsilon >0$ such that $\inf_{\theta \in F} V(\theta) \geq V_*+\epsilon$. Now we apply Theorem \ref{Thm:Pressure} in two ways: first, with the full parameter set $\Theta$ and prior $\pi_0$, and second, with $F$ in place of $\Theta$ and the conditional prior $\tilde{\pi}_0$ in place of $\pi_0$. 
Let $Z_n^F$ denote the normalizing constant in the second case.  Then for $\nu$-almost every $y \in \mathcal{Y}$, there exists $N_1 = N_1(y)$ and $N_2 = N_2(y)$ such that for all $n \geq N_1$, 
\begin{equation*}
- \frac{1}{n} \log Z_n^F(y) \geq V_* + 2\epsilon / 3,
\end{equation*}
and for all $n \geq N_2$,
\begin{equation*}
- \frac{1}{n} \log Z_n(y) \leq V_* + \epsilon/3.
\end{equation*}
Then for all $n \geq \max(N_1,N_2)$, we have
\begin{align*}
\pi_n(F \mid y) & = P_n( F \times \mathcal{X} \mid y) \\
& = \frac{1}{Z_n(y)} \int_{F \times \mathcal{X}} \exp \bigl( - \ell_n(\theta,x,y) \bigr) \, dP_0(\theta,x) \\
& = \frac{\pi_0(F) Z_n^F(y)}{Z_n(y)} \\
& \leq \exp\bigl( -V_*n - (2\epsilon/3)n + V_* n +( \epsilon/3)n \bigr) \\
& \leq \exp\bigl( - (\epsilon/3) n \bigr).
\end{align*}
Thus, for $\nu$-almost every $y \in \mathcal{Y}$, we see that $\pi_n(F \mid y)$ tends to $0$.
\end{PfofPosteriorConvergence}

\section{Posterior consistency for Gibbs processes}

Here we consider the problem of inference from direct observations of a Gibbs process, as described in Section \ref{Sect:DirectGibbs}. Recall that Gibbs processes allow one to model substantial degrees of dependence, with Markov chains of arbitrarily large order as a special case. In the present setting, we are able to establish posterior consistency (Theorem \ref{Thm:DirectGibbs}). The first step of the proof involves the application of our main results to show that the posterior distributions concentrate around $\Theta_{\min}$. Interestingly, the second main step of the proof (showing that $\Theta_{\min} = [\theta^*]$) relies on a celebrated result of Bowen about uniqueness of equilibrium states in dynamics. 

\vspace{2mm}

\begin{PfofDirectGibbs}

To begin, let us first establish the connection between the setting of Section \ref{Sect:DirectGibbs} and the general framework for Gibbs posterior inference in Section \ref{Sect:Intro}. Let $\mathcal{Z}$, $\Theta$, $\{f_{\theta} : \theta \in \Theta\}$, $\{ \mu_{\theta} : \theta \in \Theta\}$ and $\Pi_0$ be as in Section \ref{Sect:DirectGibbs}. 
In this particular application, we take $\mathcal{X}$ to be the trivial mixing SFT, which consists of exactly one point. 
Intuitively, $\mathcal{X}$ is unnecessary in this application because we make direct observations of the underlying trajectory (i.e., there is no need for an underlying ``hidden" truth). 
As $\mathcal{X}$ is trivial in this application, we omit it in our notation. Next, we let the observed system $(\mathcal{Y},T,\nu)$ be  $(\mathcal{Z},\sigma|_{\mathcal{Z}}, \mu^*)$. Then we define the loss function $\ell : \Theta \times \mathcal{Y} \to \R$ by setting $\ell(\theta,y) = \mathcal{P}(f_{\theta}) - f_{\theta}(y)$. Using our regularity assumptions on $\Theta$ and $\{f_{\theta} : \theta \in \Theta\}$, one may easily check that conditions (i)-(iii) are satisfied. We have now specified all the objects necessary for the general framework of Section \ref{Sect:Intro}. Let $\pi_0 = \Pi_0$, and for $n \geq 1$, let $\pi_n( \cdot \mid y)$ 
be the Gibbs posterior distribution on $\Theta$ given observations $(y,\dots,T^{n-1}y)$. We remind the reader that $\pi_n( \cdot \mid y)$ and $\Pi_n( \cdot \mid y_0^{n-1})$ are formally distinct distributions. Nonetheless, the following lemma shows that they are closely related. 


\begin{lemma} \label{Lemma:Equivalent}
Let $K$ be the uniform Gibbs constant for the family $\{\mu_{\theta} : \theta \in \Theta \}$. Then for any Borel set $F \subset \Theta$ and $n \geq 1$, for $y = \{y_n\} \in \mathcal{Y}$,
\begin{equation*}
K^{-2} \pi_n( F \mid y) \leq \Pi_n(F \mid y_0^{n-1}) \leq K^2 \pi_n( F \mid y)
\end{equation*}
\end{lemma}

\begin{proof}
By the uniform Gibbs property, for any $y \in \mathcal{Y}$, $\theta \in \Theta$, and $n \geq 1$, we have
\begin{equation*}
K^{-1} \exp \bigl( - \ell_n(\theta,y) \bigr) \leq \mu_{\theta}\bigl( y[0,n-1] \bigr) \leq K \exp \bigl( - \ell_n(\theta,y) \bigr).
\end{equation*}
Let $F \subset \Theta$ be a Borel set.  Integrating the inequality above with respect to $\pi_0 = \Pi_0$ yields the inequalities
\begin{align*}
K^{-1} \int_F \exp \bigl( - \ell_n(\theta,y) \bigr) \, d\pi_0(\theta) & \leq \int_F \mu_{\theta}\bigl( y[0,n-1] \bigr) \, d\Pi_0(\theta) \\
 &\leq K  \int_F  \exp \bigl( - \ell_n(\theta,y) \bigr) \, d\pi_0(\theta).
\end{align*}
Applying these upper and lower bounds to the sets $F$ and $\Theta$ we find that
\begin{align*}
K^{-2} \, \pi_n(F \mid y) & = K^{-2} \frac{\int_F \exp \bigl( - \ell_n(\theta,y) \bigr) \, d\pi_0(\theta)}{\int_{\Theta} \exp \bigl( - \ell_n(\theta,y) \bigr) \, d\pi_0(\theta)} \\
& \leq \frac{ \int_F \mu_{\theta}\bigl( y[0,n-1] \bigr) \, d\Pi_0(\theta)}{ \int_{\Theta} \mu_{\theta}\bigl( y[0,n-1] \bigr) \, d\Pi_0(\theta)} \\
& = \Pi_{n}( F \mid y_0^{n-1}),
\end{align*}
and similarly,
\begin{align*}
\Pi_n( F \mid y_0^{n-1}) & = \frac{ \int_F \mu_{\theta}\bigl( y[0,n-1] \bigr) \, d\Pi_0(\theta)}{ \int_{\Theta} \mu_{\theta}\bigl( y[0,n-1] \bigr) \, d\Pi_0(\theta)} \\
& \leq K^2 \frac{\int_F \exp \bigl( - \ell_n(\theta,y) \bigr) \, d\pi_0(\theta)}{\int_{\Theta} \exp \bigl( - \ell_n(\theta,y) \bigr) \, d\pi_0(\theta)} \\
& = K^2 \, \pi_n(F \mid y).
\end{align*}
\end{proof}

We require one additional fact before finishing the proof of the theorem. 
\begin{lemma} \label{Lemma:Darkness}
Under the present hypotheses, $\Theta_{\min} = [\theta^*]$.
\end{lemma}
\begin{proof}
Recall that $\Theta_{\min}$ is defined as the set of $\theta \in \Theta$ such that $V(\theta) = \inf \{ V(\theta') : \theta' \in \Theta\}$, where $V(\theta)$ is the rate function
\begin{equation*}
V(\theta) = \inf_{\lambda \in \mathcal{J}( S : \nu)} \biggl\{ \int \ell \, d\lambda + D(\lambda : \theta) \biggr\}.
\end{equation*}
As we have chosen $(\mathcal{X},S)$ to be trivial and $\nu = \mu_{\theta^*}$ in this application, the set of joinings $\mathcal{J}( S: \nu)$ contains only the trivial joining $\lambda = \delta_x \otimes \mu_{\theta^*}$. Hence the definition of $\ell$ ensures that
\begin{align*}
V(\theta) & = \mathcal{P}(f_{\theta}) - \int f_{\theta} \, d\mu_{\theta^*} + D( \delta_x \otimes \mu_{\theta^*} : \theta) \\
 & = \mathcal{P}(f_{\theta}) - \int f_{\theta} \, d\mu_{\theta^*}  + \biggl( \mathcal{P}(f_{\theta}) - h^{\mu_{\theta^*}}(\delta_x \otimes \mu_{\theta^*}) - \int f_{\theta} \, d\mu_{\theta^*} \biggr) \\
 & = 2 \biggl( \mathcal{P}(f_{\theta})  - \int f_{\theta} \, d\mu_{\theta^*} \biggr),
\end{align*}
where we have used that the fiber entropy of $\delta_x \otimes \mu_{\theta^*}$ over $\mu_{\theta^*}$ is trivially zero. 
To finish the proof, we will show that $V(\theta)$ is minimized if and only if $\mu_{\theta} = \mu_{\theta*}$, and therefore $\Theta_{\min} = [\theta^*]$.

First suppose that $\mu_{\theta} \neq \mu_{\theta^*}$. By the uniqueness of the Gibbs measure $\mu_{\theta}$ (see \cite{Bowen1975}) and the variational principle for pressure, we have that
\begin{equation*}
\int f_{\theta} \, d\mu_{\theta^*} + h(\mu_{\theta^*}) < \int f_{\theta} \, d\mu_{\theta} + h(\mu_{\theta}).
\end{equation*}
Subtracting $\int f_{\theta} d\mu_{\theta^*}$ from both sides, we obtain
\begin{equation*}
h(\mu_{\theta^*}) < \mathcal{P}(f_{\theta}) - \int f_{\theta} d\mu_{\theta^*}.
\end{equation*}
Then by the variational principle for pressure and this inequality, we have
\begin{align*}
V(\theta^*)/2 & = \mathcal{P}(f_{\theta^*}) - \int f_{\theta^*} d\mu_{\theta^*} \\
 & = \biggl( \int f_{\theta^*} d\mu_{\theta^*} + h(\mu_{\theta^*}) \biggr) - \int f_{\theta^*} d\mu_{\theta^*} \\
 & = h(\mu_{\theta^*}) \\
 & < \mathcal{P}(f_{\theta}) - \int f_{\theta} d\mu_{\theta^*} \\
 & = V(\theta)/2.
\end{align*}
Hence $\theta \notin \Theta_{\min}$, and we conclude that $\Theta_{\min} \subset [\theta^*]$.

Now suppose that $\mu_{\theta} = \mu_{\theta^*}$. Then by the variational principle for pressure and the fact that $\mu_{\theta} = \mu_{\theta^*}$,
\begin{align*}
V(\theta)/2 & = \mathcal{P}(f_{\theta}) - \int f_{\theta} d\mu_{\theta^*} \\
 & = \biggl( \int f_{\theta} d\mu_{\theta} + h(\mu_{\theta}) \biggr) - \int f_{\theta} d\mu_{\theta^*} \\
 & = h(\mu_{\theta}) \\
 & = h(\mu_{\theta^*}) \\
 & = \biggl( \int f_{\theta^*} d\mu_{\theta^*} + h(\mu_{\theta^*}) \biggr) - \int f_{\theta^*} d\mu_{\theta^*} \\
 & = \mathcal{P}(f_{\theta^*}) - \int f_{\theta^*} d\mu_{\theta^*} \\
 & = V(\theta^*)/2.
\end{align*}
Thus $\theta \in \Theta_{\min}$, and since $\theta \in [\theta^*]$ was arbitrary, we have shown that $\Theta_{\min} = [\theta^*]$.
%
\end{proof}

We now complete the proof of Theorem \ref{Thm:DirectGibbs}.
Let $U \subset \Theta$ be an open set such that $[\theta^*] \subset \mathcal{U}$, and let $F = \Theta \setminus U$.
By Lemma \ref{Lemma:Darkness}, we have $[\theta^*] = \Theta_{\min}$. 
Hence by Theorem \ref{Thm:PosteriorConvergence}, for $\mu^*$-almost every $y \in \mathcal{Y}$, the Gibbs posterior $\pi_n$ satisfies $\pi_n( F \mid y) \to 0$.
Then by Lemma \ref{Lemma:Equivalent}, for $\mu^*$-almost every $y \in \mathcal{Y}$, we see that the standard posterior satisfies $\Pi_n(F \mid y_0^{n-1}) \to 0$, as desired. 
\end{PfofDirectGibbs}

\section{Posterior consistency for hidden Gibbs processes}

In this section we establish posterior consistency for hidden Gibbs processes, as in Section \ref{Sect:HiddenGibbs}. In addition to modeling substantial dependence with the underlying Gibbs processes, this setting also allows for quite general observational noise models. Note that hidden Markov models with arbitrarily large order appear as a special case in this framework. Here the first part of the proof involves an application of our main results to show that the posterior converges to the set $\Theta_{\min}$. However, the second part of the proof begins with the well-known fact that the Gibbs measures $\mu_{\theta}$ satisfy large deviations principles (see \cite{Young1990}), and then relies on some recent results from \cite{McGoff2015} connecting these large deviations properties to the likelihood function in our general observational framework. 


\vspace{2mm}

\begin{PfofHiddenGibbs}

We begin by placing the setting of Section \ref{Sect:HiddenGibbs} within the general framework of Section \ref{Sect:Intro}. Let $\mathcal{X}$, $\{f_{\theta} : \theta \in \Theta\}$, $\{\mu_{\theta} : \theta \in \Theta\}$, $\Pi_0$, $\mathcal{U}$, $m$, and $\{\varphi_{\theta}( \cdot \mid x) : \theta \in \Theta, x \in \mathcal{X}\}$ be as in Section \ref{Sect:HiddenGibbs}. To define the observation space in our general framework, we let $\mathcal{Y} = \mathcal{U}^{\mathbb{N}}$. We define the map $T : \mathcal{Y} \to \mathcal{Y}$ to be the left-shift, i.e., if $y = \{y_k\} \in \mathcal{Y}$, 
then $T(y)$ is the sequence whose $k$-th coordinate is $y_{k+1}$. Furthermore, we define $\nu = \mathbb{P}_{\theta^*}^U$, which is the process measure on $\mathcal{Y}$ described in Section \ref{Sect:HiddenGibbs}. 
Then $(\mathcal{Y},T,\nu)$ is an ergodic measure preserving system (see \cite[Proposition 6.1]{McGoff2015} for ergodicity). Now define $\ell : \Theta \times \mathcal{X} \times \mathcal{Y} \to \mathbb{R}$ by $\ell(\theta,x,\{y_k\}) = - \log \varphi_{\theta}( y_0 \mid x)$. Note that the conditions (i)-(iii) on $\ell$ are satisfied by our assumptions on $\varphi$. Define $\pi_n( \cdot \mid y)$ to be the Gibbs posterior defined as in Section \ref{Sect:Intro}. 
Note that in this setting, if $y = \{y_k\}$, then the Gibbs posterior $\pi_n( \cdot \mid y)$ is equal to the standard posterior $\Pi_n( \cdot \mid y_0^{n-1})$. 
We require a few lemmas before finishing the proof of the theorem. Before we state the first such lemma, recall that $\ell^*$ denotes the $\nu$-integrable function on $\mathcal{Y}$ appearing in property (ii) in Section \ref{Sect:Inference}.

\begin{lemma} \label{Lemma:Mango}
Let $\theta \in \Theta$. Then for each $n \geq 1$ and $y \in \mathcal{Y}$, 
\begin{equation*}
\biggl| \frac{1}{n} \log \int_{\mathcal{X}} \exp\bigl( - \ell_n(\theta,x,y) \bigr) \, d\mu_{\theta}(x) \biggr| \leq \frac{1}{n} \sum_{k=0}^{n-1} \ell^* \bigl(T^ky\bigr).
\end{equation*}
\end{lemma}

\begin{proof}
For notation, let
\begin{equation*}
I_n(y) = \int_{\mathcal{X}} \exp\bigl( - \ell_n(\theta,x,y) \bigr) \, d\mu_{\theta}(x).
\end{equation*}
First suppose that $I_n(y) \leq 1$. Then by Jensen's inequality and the definition of $\ell^*$,
\begin{align*}
\biggl| \frac{1}{n} \log I_n(y) \biggr| & = - \frac{1}{n} \log \int_{\mathcal{X}} \exp\bigl( - \ell_n(\theta,x,y) \bigr) \, d\mu_{\theta}(x) \\
& \leq  \frac{1}{n} \int \ell_n(\theta,x,y) \, d\mu_{\theta}(x) \\
& =  \frac{1}{n} \sum_{k=0}^{n-1}  \int \ell(\theta,S^kx,T^ky) \, d\mu_{\theta}(x) \\
& \leq  \frac{1}{n} \sum_{k=0}^{n-1} \ell^*\bigl(T^ky \bigr).
\end{align*}

Now suppose that $I_n(y) > 1$. Then 
\begin{align*}
\biggl| \frac{1}{n} \log I_n(y) \biggr| & = \frac{1}{n} \log \int_{\mathcal{X}} \exp\bigl( - \ell_n(\theta,x,y) \bigr) \, d\mu_{\theta}(x) \\
& \leq \frac{1}{n} \log \,  \sup_{x \in \mathcal{X}} \, \exp\bigl( - \ell_n(\theta,x,y) \bigr) \\
& \leq \frac{1}{n} \sum_{k=0}^{n-1} \sup_{x \in \mathcal{X}} |\ell(\theta,S^kx,T^k(y))| \\
& \leq  \frac{1}{n} \sum_{k=0}^{n-1} \ell^*\bigl(T^ky\bigr),
\end{align*}
where we have used that both the logarithm and the exponential are increasing. 
\end{proof}

\begin{lemma} \label{Lemma:Peach}
Let $\theta \in \Theta$. Then
\begin{equation*}
\lim_n - \frac{1}{n} \mathbb{E}_{\theta^*} \Bigl[ \log p_{\theta}\bigl(Y_0^{n-1}\bigr) \Bigr] = V(\theta).
\end{equation*}
\end{lemma}
\begin{proof}
For each $n \geq 1$, let
\begin{equation*}
f_n(y) = - \frac{1}{n} \log \int_{\mathcal{X}} \exp \bigl( - \ell_n(\theta,x,y) \bigr) \, d\mu_{\theta}(x),
\end{equation*}
and let $F_n(y) = n^{-1} \sum_{k=0}^{n-1} \ell^*(T^ky)$. 
By property (ii), $\ell^*$ is $\nu$-integrable and thus 
the pointwise ergodic theorem ensures that $F_n(y)$ converges for $\nu$-almost every $y$ to the constant 
$\mathbb{E}_{\theta^*}[ \ell^*]$. Furthermore, $\lim_n \mathbb{E}_{\theta^*}[F_n] = \mathbb{E}_{\theta^*}[ \ell^*]$. 
By Lemma \ref{Lemma:Mango}, $|f_n| \leq F_n$ for each $n \geq 1$.  Therefore, by the generalized Lebesgue dominated convergence 
theorem and the definition of the loss,
\begin{align*}
\lim_n - \frac{1}{n} \mathbb{E}_{\theta^*} \Bigl[ \log p_{\theta}\bigl(Y_0^{n-1}\bigr) \Bigr] & = \lim_n \mathbb{E}_{\theta^*} [f_n] \\
 & = \mathbb{E}_{\theta^*} \Bigl[ \lim_n f_n \Bigr].
\end{align*}
By Theorem \ref{Thm:Pressure}, the $\mathbb{P}_{\theta^*}$-almost sure limit of $\{f_n\}$ is equal to $V(\theta)$. Combining these facts, we obtain the desired equality. 
\end{proof}

\begin{lemma} \label{Lemma:Watermelon}
Suppose $\theta \in \Theta \setminus [\theta^*]$. Then
\begin{equation*}
V(\theta^*) < \lim_n - \frac{1}{n} \mathbb{E}_{\theta^*} \Bigl[ \log p_{\theta}\bigl(Y_0^{n-1}\bigr) \Bigr].
\end{equation*}
\end{lemma}

\begin{proof}
The well-known large deviations principles for the Gibbs measures $\mu_{\theta}$ \cite{Young1990} imply that they satisfy property (L1) from \cite{McGoff2015}. By hypothesis, $g$ satisfies the regularity of observations property (L2) from \cite{McGoff2015}. Then results from \cite{McGoff2015} (in particular Propositions 4.3 and 6.4) yield the desired inequality.
\end{proof}

We now proceed with the proof of Theorem \ref{Thm:HiddenGibbs}.
Recall that for $y = \{y_k\} \in \mathcal{Y}$ our choice of loss function ensures that 
the Bayesian posterior $\Pi_n( \cdot \mid y_0^{n-1})$ is equal to the Gibbs posterior $\pi_n( \cdot \mid y)$. 
By Theorem \ref{Thm:PosteriorConvergence}, the Gibbs posterior $\pi_n( \cdot \mid Y)$ concentrates $\nu$-almost surely around the set $\Theta_{\min}$, defined as the set of $\theta \in \Theta$ such that $V(\theta) = \inf \{ V(\theta') : \theta' \in \Theta \}$. Hence $\Pi_n( \cdot \mid Y_0^{n-1})$ concentrates $\mathbb{P}_{\theta^*}^U$-almost surely around $\Theta_{\min}$. It remains to show that $\Theta_{\min} = [\theta^*]$.

Suppose $\theta \in \Theta \setminus [\theta^*]$. Then by Lemmas \ref{Lemma:Peach} and \ref{Lemma:Watermelon}, we have
\begin{equation*}
V(\theta) = \lim_n - \frac{1}{n} \mathbb{E}_{\theta^*} \Bigl[ \log p_{\theta}\bigl(Y_0^{n-1}\bigr) \Bigr] > V(\theta^*).
\end{equation*}
It follows immediately that $\Theta_{\min} \subset [\theta^*]$. For the reverse inclusion, note that if $\theta \in [\theta^*]$, then $\mathbb{P}_{\theta}^U = \mathbb{P}_{\theta^*}^U$, and thus for each $n$, 
\begin{equation*}
\mathbb{E}_{\theta^*} \biggl[ \log p_{\theta} \bigl(Y_0^{n-1} \bigr) \biggr] = \mathbb{E}_{\theta^*} \biggl[ \log p_{\theta^*} \bigl(Y_0^{n-1} \bigr) \biggr].
\end{equation*}
Then Lemma \ref{Lemma:Peach} gives that $V(\theta) = V(\theta^*)$ for each $\theta \in [\theta^*]$. 
This concludes the proof of Theorem \ref{Thm:HiddenGibbs}.
\end{PfofHiddenGibbs}

\section{Additional results}

In this section we collect some auxiliary results about Gibbs posterior inference. 
We begin with a converse to Theorem \ref{Thm:PosteriorConvergence} on the exponential scale: 
if $U$ is an open set intersecting $\Theta_{\min}$, 
then the Gibbs posterior measure of $U$ cannot be exponentially small as $n$ tends to infinity.

\begin{proposition} \label{Prop:TakeABite}
Suppose $U \subset \Theta$ is open and $U \cap \Theta_{\min} \neq \varnothing$. Then for $\nu$-almost every $y \in \mathcal{Y}$,
\begin{equation*}
\lim_n \frac{1}{n} \log \pi_n( U \mid y) = 0.
\end{equation*}
\end{proposition}
\begin{proof}
Let $\theta_0 \in U \cap \Theta_{\min}$.  By definition of $\Theta_{\min}$ we have 
$V(\theta_0) = V_* = \inf_{\theta} V(\theta)$. 
Fix $\epsilon > 0$ and select $\delta > 0$ sufficiently small that $\int \rho_{\delta} \, d\nu < \epsilon$ and that 
the ball $U_0$ of radius $\delta$ around $\theta_0$ is contained in $U$. 
Since $\pi_0$ is fully supported, $\pi_0(U_0) >0$. 
Note that for each $y \in \mathcal{Y}$ and $n \geq 1$, 
\begin{align*}
\pi_n( U \mid y) & = \frac{1}{Z_n(y)} \int_U \int_{\mathcal{X}} \exp \bigl( - \ell_n(\theta,x,y) \bigr) \, d\mu_{\theta}(x) d\pi_0(\theta) \\
 & \geq \frac{1}{Z_n(y)} \exp \biggl( - \sum_{k = 0}^{n-1} \rho_{\delta}( T^k y) \biggr) \int_{\mathcal{X}} \exp \bigl( - \ell_n(\theta_0,x,y) \bigr) \, d\mu_{\theta_0}(x) \cdot \pi_0(U_0).
\end{align*}
Taking logarithms, dividing by $n$, and letting $n$ tend to infinity yields
\begin{equation*}
\liminf_{n} \frac{1}{n} \log \pi_n( U \mid y) \, \geq \, V_* - \int \rho_{\delta} \, d\nu - V_* \geq - \epsilon.
\end{equation*}
As $\epsilon >0$ was arbitrary, we obtain the desired result.
\end{proof}

We now address the Ces\`{a}ro convergence of the full posterior $P_n$ on $\Theta \times \mathcal{X}$. 
Recall that we let $\Id : \Theta \to \Theta$ be the identity map on $\Theta$. 
In the thermodynamic formalism, invariant measures that achieve the optimal value in the variational expression for pressure are called equilibrium measures.
In our setting, we introduce terminology for joinings that achieve the optimal value in the variational expression for the rate function.
We will call a joining $\lambda \in \mathcal{J}( \Id \times S : \nu)$ an equilibrium joining if 
\begin{equation*}
\lambda \in \argmin \biggl\{ \int \ell \, d\lambda' + G(\lambda') : \lambda' \in \mathcal{J}( \Id \times S : \nu) \biggr\}.
\end{equation*}

\begin{proposition} \label{Prop:CesaroConvergence}
For each $y \in \mathcal{Y}$ and $n \geq 1$, let $Q_n( \cdot \mid y) \in \mathcal{M}( \Theta \times \mathcal{X})$ be defined for Borel sets $E \subset \Theta \times \mathcal{X}$ by
\begin{equation*}
Q_n( E \mid y ) = \frac{1}{n} \sum_{k=0}^{n-1} P_n( (\Id \times S)^{-k} E  \mid y).
\end{equation*}
Then for $\nu$-almost every $y \in \mathcal{Y}$, all limit points of $\{Q_n( \cdot \mid y) \}_{n \geq 1}$ are $(\Theta \times \mathcal{X})$-marginals of equilibrium joinings.
\end{proposition}

\begin{proof}
As in Section \ref{Sect:UB}, let
\begin{equation*}
\eta_n = \frac{1}{n} \sum_{k=0}^{n-1} (P_n^y \circ (\Id \times S)^{-k}) \otimes \delta_{T^ky}.
\end{equation*}
By definition, $Q_n( \cdot \mid y)$ is the $(\Theta \times \mathcal{X})$-marginal of $\eta_n$. 
Let $Q$ be a weak limit of the subsequence $\{Q_{n_k}( \cdot \mid y) \}_{k \geq 1}$. 
By repeating the arguments of Section \ref{Sect:UB}, one may show that there is a subsequence 
$\{n_{k_j}\}_{j \geq 1}$ such that $\{\eta_{n_{k_j}}\}_{j \geq 1}$ converges weakly to an equilibrium joining 
$\lambda$.  As $Q$ is necessarily the $(\Theta \times \mathcal{X})$-marginal of the limit
$\lambda$, the proof is complete. 
\end{proof}

\bibliographystyle{plain}
\bibliography{Gibbs_refs}

\begin{thebibliography}{10}

\bibitem{Allahbakhshi2013}
Mahsa Allahbakhshi and Anthony Quas.
\newblock Class degree and relative maximal entropy.
\newblock {\em Transactions of the American Mathematical Society},
  365(3):1347--1368, 2013.

\bibitem{Allahbakhshi2017}
Masha Allahbakhshi, John Antonioli, and Jisang Yoo.
\newblock Relative equilibrium states and class degree.
\newblock {\em Ergodic Theory and Dynamical Systems}, pages 1--24, 2017.

\bibitem{Alves2018}
Jos\'{e}~F. Alves, Vanessa Ramos, and Jaqueline Siqueira.
\newblock Equilibrium stability for non-uniformly hyperbolic systems.
\newblock {\em Ergodic Theory and Dynamical Systems}, pages 1--24, 2018.

\bibitem{Antonioli2016}
John Antonioli.
\newblock Compensation functions for factors of shifts of finite type.
\newblock {\em Ergodic Theory and Dynamical Systems}, 36(2):375--389, 2016.

\bibitem{Benedicks2000}
Michael Benedicks and Lai-Sang Young.
\newblock Markov extensions and decay of correlations for certain {H}{\'e}non
  maps.
\newblock {\em Ast{\'e}risque}, 261:13--56, 2000.

\bibitem{Bissiri2016}
Pier~Giovanni Bissiri, Chris~C Holmes, and Stephen~G Walker.
\newblock A general framework for updating belief distributions.
\newblock {\em Journal of the Royal Statistical Society: Series B (Statistical
  Methodology)}, 78(5):1103--1130, 2016.

\bibitem{Bowen1975}
Rufus Bowen.
\newblock {\em Equilibrium States and the Ergodic Theory of Anosov
  Diffeomorphisms}, volume 470.
\newblock Springer, Berlin, Heidelberg, 1975.

\bibitem{Chazottes1998}
J-R Chazottes, E~Floriani, and R~Lima.
\newblock Relative entropy and identification of {G}ibbs measures in dynamical
  systems.
\newblock {\em Journal of Statistical Physics}, 90(3-4):697--725, 1998.

\bibitem{Chopinetal2015}
Nicolas Chopin, S\'ebastien Gadat, Benjamin Guedj, Arnaud Guyader, and Elodie
  Vernet.
\newblock On some recent advances on high dimensional {B}ayesian statistics.
\newblock {\em ESAIM: Proceedings and Surveys}, 51:293--319, 2015.

\bibitem{Cover2012}
Thomas~M Cover and Joy~A Thomas.
\newblock {\em Elements of information theory}.
\newblock John Wiley \& Sons, 2012.

\bibitem{Rue2006}
Thierry de~la Rue.
\newblock {An Introduction to Joinings in Ergodic Theory}.
\newblock {\em Discrete and Continuous Dynamical Systems}, 15(1):121--142,
  2006.

\bibitem{diaconis1998}
Persi~W. Diaconis and David Freedman.
\newblock Consistency of {B}ayes estimates for nonparametric regression: normal
  theory.
\newblock {\em Bernoulli}, 4(4):411--444, 12 1998.

\bibitem{Doob49CNRS}
J.~L. Doob.
\newblock Application of the theory of martingales.
\newblock {\em Colloques Internationaux du Centre National de la Recherche
  Scientifique}, pages 23--27, 1949.

\bibitem{Doucetal2016}
Randal Douc, Jimmy Olsson, and Francois Roueff.
\newblock Posterior consistency for partially observed {M}arkov models.
\newblock {\em arXiv preprint arXiv:1608.06851}, 2016.

\bibitem{Dupuis2011}
Paul Dupuis and Richard~S Ellis.
\newblock {\em {A Weak Convergence Approach to the Theory of Large
  Deviations}}, volume 902.
\newblock John Wiley \& Sons, 2011.

\bibitem{Furstenberg1967}
Harry Furstenberg.
\newblock Disjointness in ergodic theory, minimal sets, and a problem in
  {D}iophantine approximation.
\newblock {\em Theory of Computing Systems}, 1(1):1--49, 1967.

\bibitem{Gassiat2014}
Elisabeth Gassiat and Judith Rousseau.
\newblock About the posterior distribution in hidden {M}arkov models with
  unknown number of states.
\newblock {\em Bernoulli}, 20(4):2039--2075, 2014.

\bibitem{GemanGeman1984}
Stuart Geman and Donald Geman.
\newblock {Stochastic Relaxation, Gibbs Distributions, and the Bayesian
  Restoration of Images}.
\newblock {\em IEEE Transactions on Pattern Analysis and Machine Intelligence},
  6(6):721--741, November 1984.

\bibitem{Ghosal2017}
Subhashis Ghosal and Aad Van~der Vaart.
\newblock {\em {Fundamentals of Nonparametric Bayesian Inference}}, volume~44.
\newblock Cambridge University Press, 2017.

\bibitem{Glasner2003}
Eli Glasner.
\newblock {\em {Ergodic Theory Via Joinings}}, volume 101.
\newblock American Mathematical Soc., 2003.

\bibitem{Hang2017}
Hanyuan Hang and Ingo Steinwart.
\newblock {A Bernstein-type inequality for some mixing processes and dynamical
  systems with an application to learning}.
\newblock {\em The Annals of Statistics}, 45(2):708--743, 2017.

\bibitem{Huber1981}
Peter~.J. Huber.
\newblock {\em {Robust Statistics}}.
\newblock Wiley New York, 1981.

\bibitem{Jaynes1968}
Edwin~T Jaynes.
\newblock {Prior Probabilities}.
\newblock {\em IEEE Transactions on Systems Science and Cybernetics}, 4(3):227,
  1968.

\bibitem{Jaynes1973-JAYTWP}
Edwin~T Jaynes.
\newblock The well-posed problem.
\newblock {\em Foundations of Physics}, 3(4):477--493, 1973.

\bibitem{Jiang2008}
Wenxin Jiang and Martin~A Tanner.
\newblock Gibbs posterior for variable selection in high-dimensional
  classification and data mining.
\newblock {\em The Annals of Statistics}, pages 2207--2231, 2008.

\bibitem{Kifer2001}
Yuri Kifer.
\newblock On the topological pressure for random bundle transformations.
\newblock In {\em Topology, ergodic theory, real algebraic geometry}, volume
  202 of {\em American Mathematical Society Translations. Series 2}, pages
  197--214. Amer. Math. Soc., Providence, RI, 2001.

\bibitem{Kifer2006}
Yuri Kifer and Pei-Dong Liu.
\newblock Random dynamics.
\newblock {\em Handbook of dynamical systems}, 1:379--499, 2006.

\bibitem{Lalley1999}
Steven~P. Lalley.
\newblock Beneath the noise, chaos.
\newblock {\em The Annals of Statistics}, 27(2):461--479, 1999.

\bibitem{LalleyNobel2006}
Steven~P. Lalley and A.~B. Nobel.
\newblock Denoising deterministic time series.
\newblock {\em Dynamics of Partial Differential Equations}, 3(4):259--279,
  2006.

\bibitem{Law2015}
Kody Law, Andrew Stuart, and Kostas Zygalakis.
\newblock {\em Data Assimilation}.
\newblock Cham, Switzerland: Springer, 2015.

\bibitem{LedrappierWalters1977}
Fran{\c{c}}ois Ledrappier and Peter Walters.
\newblock A relativised variational principle for continuous transformations.
\newblock {\em Journal of the London Mathematical Society}, 2(3):568--576,
  1977.

\bibitem{LindMarcus}
Douglas Lind and Brian Marcus.
\newblock {\em An introduction to symbolic dynamics and coding}.
\newblock Cambridge university press, 1995.

\bibitem{Luetal}
Yulong Lu, Andrew Stuart, and Hendrik Weber.
\newblock {Gaussian approximations for transition paths in Brownian dynamics}.
\newblock {\em SIAM Journal on Mathematical Analysis}, 49(4):3005Ð3047, 2017.

\bibitem{McGoff2015}
Kevin McGoff, Sayan Mukherjee, Andrew Nobel, and Natesh Pillai.
\newblock Consistency of maximum likelihood estimation for some dynamical
  systems.
\newblock {\em The Annals of Statistics}, 43(1):1--29, 2015.

\bibitem{McGoffSurvey2015}
Kevin McGoff, Sayan Mukherjee, and Natesh Pillai.
\newblock Statistical inference for dynamical systems: A review.
\newblock {\em Statistics Surveys.}, 9:209--252, 2015.

\bibitem{McGoff2016variational}
Kevin McGoff and Andrew Nobel.
\newblock Variational analysis of inference from dynamical systems.
\newblock {\em arXiv preprint arXiv:1601.05033}, 2016.

\bibitem{McGoff2016empirical}
Kevin McGoff and Andrew~B Nobel.
\newblock Empirical risk minimization and complexity of dynamical models.
\newblock {\em arXiv preprint arXiv:1611.06173}, 2016.

\bibitem{Mitter2}
Nigel~J Newton and Sanjoy~K Mitter.
\newblock Variational approach to nonlinear estimation.
\newblock {\em SIAM Journal on Control and Optimization}, 42(5):1813Ð1833,
  2003.

\bibitem{1742-5468-2012-11-P11008}
Nigel~J Newton and Sanjoy~K Mitter.
\newblock {Variational Bayes and a problem of reliable communication: II.
  Infinite systems}.
\newblock {\em Journal of Statistical Mechanics: Theory and Experiment},
  2012(11):P11008, 2012.

\bibitem{Petersen2003}
Karl Petersen, Anthony Quas, and Sujin Shin.
\newblock Measures of maximal relative entropy.
\newblock {\em Ergodic Theory and Dynamical Systems}, 23(1):207--223, 2003.

\bibitem{Ruelle2004}
David Ruelle.
\newblock {\em Thermodynamic formalism: the mathematical structure of
  equilibrium statistical mechanics}.
\newblock Cambridge University Press, 2004.

\bibitem{Sarig2008}
Omri Sarig.
\newblock {\em Lecture Notes on Ergodic Theory}.
\newblock 2008.

\bibitem{Schwartz1965}
Lorraine Schwartz.
\newblock On {B}ayes procedures.
\newblock {\em Zeitschrift f{\"u}r Wahrscheinlichkeitstheorie und verwandte
  Gebiete}, 4(1):10--26, 1965.

\bibitem{Steinwart2009}
Ingo Steinwart and Marian Anghel.
\newblock Consistency of support vector machines for forecasting the evolution
  of an unknown ergodic dynamical system from observations with unknown noise.
\newblock {\em The Annals of Statistics}, pages 841--875, 2009.

\bibitem{Vernet2015}
Elodie Vernet.
\newblock Posterior consistency for nonparametric hidden {M}arkov models with
  finite state space.
\newblock {\em Electronic Journal of Statistics}, 9(1):717--752, 2015.

\bibitem{Walters1982}
Peter Walters.
\newblock {\em An introduction to ergodic theory}, volume~79.
\newblock Springer Science \& Business Media, 1982.

\bibitem{Walters1986}
Peter Walters.
\newblock Relative pressure, relative equilibrium states, compensation
  functions and many-to-one codes between subshifts.
\newblock {\em Transactions of the American Mathematical Society},
  296(1):1--31, 1986.

\bibitem{Young1990}
Lai-Sang Young.
\newblock Large deviations in dynamical systems.
\newblock {\em Transactions of the American Mathematical Society},
  318(2):525--543, 1990.

\bibitem{Young1998}
Lai-Sang Young.
\newblock Statistical properties of dynamical systems with some hyperbolicity.
\newblock {\em Annals of Mathematics}, 147(3):585--650, 1998.

\bibitem{Young1999}
Lai-Sang Young.
\newblock Recurrence times and rates of mixing.
\newblock {\em Israel Journal of Mathematics}, 110(1):153--188, 1999.

\bibitem{Zellner1988}
Arnold Zellner.
\newblock {Optimal Information Processing and Bayes's Theorem}.
\newblock {\em The American Statistician}, 42(4):278--280, 1988.

\bibitem{Zhang2006a}
Tong Zhang.
\newblock From $\epsilon$-entropy to {KL}-entropy: {A}nalysis of minimum
  information complexity density estimation.
\newblock {\em The Annals of Statistics}, 34(5):2180--2210, 2006.

\bibitem{Zhang2006b}
Tong Zhang.
\newblock Information-theoretic upper and lower bounds for statistical
  estimation.
\newblock {\em IEEE Transactions on Information Theory}, 52(4):1307--1321,
  2006.

\end{thebibliography}

\end{document}